\documentclass{amsart}
\usepackage{amssymb}
\usepackage{amsmath}
\usepackage{amsfonts}
\usepackage{graphicx,color}

\newtheorem{theorem}{Theorem}
\theoremstyle{plain}
\newtheorem{acknowledgement}{Acknowledgements}

\newtheorem{lemma}{Lemma}

\newtheorem{proposition}{Proposition}

\numberwithin{equation}{section}

\begin{document}
\title[Singular $(p(x),q(x))$- Laplacian systems ]{Existence and regularity
of solutions for a class of singular $(p(x),q(x))$- Laplacian systems }

\begin{abstract}
In this paper we study the existence of positive smooth solutions for a
class of singular $(p(x),q(x))$- Laplacian systems by using sub and
supersolution methods.
\end{abstract}

\author{Claudianor O. Alves}
\address{Claudianor O. Alves\\
Universidade Federal de Campina Grande, Unidade Acad\^{e}mica de Matem\'{a}
tica, CEP:58429-900, Campina Grande - PB, Brazil.}
\email{coalves@mat.ufcg.edu.br}
\thanks{C.O. Alves was partially supported by CNPq/Brazil 304036/2013-7 and
INCT-MAT}
\author{Abdelkrim Moussaoui}
\address{Abdelkrim Moussaoui\\
Biology Department, A. Mira Bejaia University, Targa Ouzemour, 06000 Bejaia,
Algeria.}
\email{abdelkrim.moussaoui@univ-bejaia.dz}
\thanks{A. Moussaoui was supported by CNPq/Brazil 402792/2015-7.}
\maketitle

\section{Introduction}

\label{sec1}

In the present paper, we study the existence of solution for the following
class of singular $(p(x),q(x))$-Laplacian equations
\begin{equation}
\left\{
\begin{array}{ll}
-\Delta _{p(x)}u=\lambda u^{\alpha _{1}(x)}v^{\beta _{1}(x)} & \text{in }%
\Omega , \\
-\Delta _{q(x)}v=\lambda u^{\alpha _{2}(x)}v^{\beta _{2}(x)} & \text{in }%
\Omega , \\
u,v>0 & \text{in }\Omega , \\
u,v=0 & \text{on }\partial \Omega ,%
\end{array}%
\right.  \tag{$P$}  \label{p}
\end{equation}%
where $\Omega $ is a bounded domain in $%
%TCIMACRO{\U{211d} }%
%BeginExpansion
\mathbb{R}
%EndExpansion
^{N}$ $\left( N\geq 2\right) $ with $C^{2}$ boundary $\partial \Omega $ and $%
\lambda >0$ is a parameter. Here, $\Delta _{p(x)}$ and $\Delta _{q(x)}$
stand for the $p(x)$-Laplacian and $q(x)$-Laplacian operators respectively,
that is,
\begin{equation*}
\Delta _{p(x)}u=div(\left\vert \nabla u\right\vert ^{p(x)-2}\nabla u)\quad %
\mbox{and}\quad \Delta _{q(x)}v=div(\left\vert \nabla v\right\vert
^{q(x)-2}\nabla v)
\end{equation*}%
with $p,q\in C^{1}(\overline{\Omega })$ and
\begin{equation}
\begin{array}{l}
1<p^{-}\leq p^{+}<N\text{ \ and \ }1<q^{-}\leq q^{+}<N.%
\end{array}
\label{33}
\end{equation}%
Thought out this paper, we denote by%
\begin{equation*}
\begin{array}{l}
s^{-}=\displaystyle\inf_{x\in \Omega }s(x)\text{ \ and \ }s^{+}=\displaystyle%
\sup_{x\in \Omega }s(x).%
\end{array}%
\end{equation*}

A solution of (\ref{p}) is understood in the weak sense, that is, a pair $%
(u,v)\in W_{0}^{1,p(x)}(\Omega )\times W_{0}^{1,q(x)}(\Omega )$, where $u,v$
are positive in $\Omega $ and satisfy
\begin{equation}
\left\{
\begin{array}{cc}
\int_{\Omega }|\nabla u|^{p(x)-2}\nabla u\nabla \varphi \ dx & =\int_{\Omega
}u^{\alpha _{1}(x)}v^{\beta _{1}(x)}\varphi \ dx \\
\int_{\Omega }|\nabla v|^{q(x)-2}\nabla v\nabla \psi \ dx & =\int_{\Omega
}u^{\alpha _{2}(x)}v^{\beta _{2}(x)}\psi \ dx%
\end{array}%
\right.  \label{7}
\end{equation}%
for all $(\varphi ,\psi )\in W_{0}^{1,p(x)}(\Omega )\times
W_{0}^{1,q(x)}(\Omega )$.

The main interest of this work is that the nonlinearities in the right hand
side of equations in (\ref{p}) can exhibit singularities when the variables $%
u$ and $v$ approach zero. This occur through the variable exponents which
are allowed to be negative. In this context, we will consider two situations
regarding the structure of system (\ref{p}):
\begin{equation}
\alpha _{2}^{-},\beta _{1}^{-}>0\text{ \ (cooperative structure)}  \label{h3}
\end{equation}
and
\begin{equation}
\alpha _{2}^{+},\beta _{1}^{+}<0\text{ \ (competitive structure).}
\label{h4}
\end{equation}
For system (\ref{p}) associated with (\ref{h3}), the right term in the first
(resp. second) equation of (\ref{p}) is increasing in $v$ (resp. $u$), which
do not occur for (\ref{p}) under (\ref{h4}). In addition of (\ref{h3}), we
assume
\begin{equation}
\begin{array}{l}
\alpha _{2}^{+}<p^{-}-1,\text{ \ \ }\beta _{1}^{+}<q^{-}-1,\text{ \ \ }%
\alpha _{1}^{+},\beta _{2}^{+}<0\text{ \ \ and \ }\alpha _{1}^{-},\beta
_{2}^{-}>-1/N.%
\end{array}
\label{h1}
\end{equation}%
For (\ref{p}) under competitive structure (\ref{h4}), we also assume
assumptions:
\begin{equation}
\left\{
\begin{array}{l}
0>\alpha _{1}^{+}\geq \alpha _{1}^{-}>\max \{-\frac{1}{N},-(p^{-}-1)\} \\
0>\beta _{2}^{+}\geq \beta _{2}^{-}>\max \{-\frac{1}{N},-(q^{-}-1)\}.%
\end{array}%
\right.  \label{h2}
\end{equation}

This type of problem is rare in the literature. Actually, according to our
knowledge, singular system (\ref{p}) was examined only when the exponent
variable functions $p(\cdot ),q(\cdot ),\alpha _{i}(\cdot )$ and $\beta
_{i}(\cdot )$, $i=1,2$, are reduced to be constants. In this case, $\Delta
_{p(x)}$ and $\Delta _{q(x)}$become the well-known $p$-Laplacian and $q$%
-Laplacian operators. For a complete overview on the study of the constant
exponent case, we refer to \cite{AM, dkm, MM}\ for system (\ref{p}) with
cooperative structure, while we quote \cite{MM2, MM3} for the study of
competitive structure in (\ref{p}).

The $p(x)$-Laplacian operator possesses more complicated nonlinearity than
the $p$-Laplacian. For instance, it is inhomogeneous and in general, it has
no first eigenvalue, that is, the infimum of the eigenvalues of $p(x)$%
-Laplacian equals $0$ (see \cite{FZZ2}). Thus, transposing the results
obtained with the p-Laplacian to the problems arising the $p(x)$-Laplacian
operator is not easy task. The study of these problems are often very
complicated and require relevant topics of nonlinear functional analysis,
especially the theory of variable exponent Lebesgue and Sobolev spaces (see,
e.g., \cite{dhhr} and its abundant reference).

Partial differential equations involving the $p(x)$-Laplacian arise, for
instance, as a mathematical model for problems involving electrorheological
fluids and image restorations, see \cite%
{Acerbi1,Acerbi2,Antontsev,CLions,Chen,Ru}. This explains the intense
research on this subject in the last decades, see for example the papers
\cite%
{Alves2,Alves3,AlvesBarreiro,AlvesFerreira,AlvesFerreira1,AlvesSouto,Fan1,Fan,FanZhao0,FZ1,FZZ, FanShenZhao,BSS,BSS1,FuZa,MR, QZ}
and their references.

The main results of the present paper provide the existence and regularity
of (positive) solutions for problem (\ref{p}) under assumptions (\ref{h3})
and (\ref{h4}). Our first result is related to cooperative case and it is
formulated as follows.

\begin{theorem}
\label{T1} Under assumptions (\ref{h3}) and (\ref{h1}), system (\ref{p}) has
a positive solution $\left( u,v\right) $ in $C^{1,\nu }(\overline{\Omega }%
)\times C^{1,\nu }(\overline{\Omega }),$ for certain $\nu \in (0,1)$ and $%
\lambda >0$ large. Moreover, there exists a constant $c>0$ such that%
\begin{equation*}
u(x),v(x)\geq cd(x)\text{ \ as }x\rightarrow \partial \Omega ,
\end{equation*}%
where $d(x):=dist(x,\partial \Omega ).$
\end{theorem}

The second main result deals with the competitive structure and it has the
following statement.

\begin{theorem}
\label{T12} Assume (\ref{h4}) and (\ref{h2}) hold with%
\begin{equation}
\begin{array}{l}
\alpha _{1}^{-}+\beta _{1}^{-}>-\frac{1}{N}\text{ \ and \ }\alpha
_{2}^{-}+\beta _{2}^{-}>-\frac{1}{N}.%
\end{array}
\label{h4**}
\end{equation}%
Then, system (\ref{p}) has a positive solution $\left( u,v\right) $ in $%
C^{1,\nu }(\overline{\Omega })\times C^{1,\nu }(\overline{\Omega })$, for $%
\nu \in (0,1)$ and $\lambda >0$ large. Moreover, there exists a constant $%
c^{\prime }>0$ such that%
\begin{equation*}
u(x),v(x)\geq c^{\prime }d(x)\text{ \ as }x\rightarrow \partial \Omega ,
\end{equation*}
\end{theorem}

The proofs of Theorems \ref{T1} and \ref{T12} are chiefly based on Theorems %
\ref{T2} and \ref{T3} stated in Section \ref{sec3}, respectively, which are
a version of the sub-supersolution method for quasilinear singular elliptic
systems involving variable exponents. They are shown via Schauder's fixed
point theorem together with adequate truncations. It is worth pointing out
that in these Theorems no sign condition is required on the right-hand side
nonlinearities and so they can be used for large classes of quasilinear
singular problems involving $p(x)$-Laplacian operator. However, due to
competitive structure of the problem in Theorem \ref{T3}, the nonlinearities
are required to be more regular in order to offset the loss of the
monotonicity. A significant feature of our result lies in the obtaining of
the sub- and supersolution. This is achieved by the choice of suitable
functions with an adjustment of adequate constants.

Another important point discussed in this paper concerns the regularity of
solutions for singular problems involving $p(x)$-Laplacian operator.
According to our knowledge, this topic is a novelty. We emphasize that the
regularity result is crucial in the proof of Theorems \ref{T2} and \ref{T3},
besides ensuring the smoothness of the obtained solutions of problem (\ref{p}%
) in Theorems \ref{T1} and \ref{T12}.

The plan of the paper is as follows: In Section \ref{sec2} we prove some
technical results. In Section \ref{sec3} we show two general results which
will be used in the proof of our main results while in Sections \ref{sec4}
and \ref{sec5} we prove the Theorems \ref{T1} and \ref{T12} respectively.

\section{Technical results}

\label{sec2}

Let $L^{p(x)}(\Omega )$ be the generalized Lebesgue space that consists of
all measurable real-valued functions $u$ satisfying%
\begin{equation*}
\begin{array}{l}
\rho _{p(x)}(u)=\int_{\Omega }|u(x)|^{p(x)}dx<+\infty ,%
\end{array}%
\end{equation*}%
endowed with the Luxemburg norm%
\begin{equation*}
\begin{array}{l}
\left\Vert u\right\Vert _{p(x)}=\inf \{\tau >0:\rho _{p(x)}(\frac{u}{\tau }%
)\leq 1\}.%
\end{array}%
\end{equation*}%
The variable exponent Sobolev space $W_{0}^{1,p(\cdot )}(\Omega )$ is
defined by%
\begin{equation*}
\begin{array}{l}
W_{0}^{1,p(x)}(\Omega )=\{u\in L^{p(x)}(\Omega ):|\nabla u|\in
L^{p(x)}(\Omega )\}.%
\end{array}%
\end{equation*}%
The norm $\left\Vert u\right\Vert _{1,p(x)}=\left\Vert \nabla u\right\Vert
_{p(x)}$ makes $W_{0}^{1,p(x)}(\Omega )$ a Banach space, for more details
see \cite{FZ}. In the sequel, corresponding to $1<p(x)<+\infty $, we denote $%
p(x)^{\prime }=\frac{p(x)}{p(x)-1}$.\newline

In \cite[Lemma 3.2]{QZ}, Zhang has proved that there are $\delta ,\lambda
_{0}>0$ such that function
\begin{equation*}
w(x)=\left\{
\begin{array}{l}
d(x),\,d(x)<\delta , \\
\delta +\int_{\delta }^{d(x)}\left( \frac{\delta -t}{\delta }\right) ^{\frac{%
2}{p^{-}1}},\,\delta \leq d(x)\leq 2\delta, \\
\delta +\int_{\delta }^{2\delta }\left( \frac{\delta -t}{\delta }\right) ^{%
\frac{2}{p^{-}1}},\,\delta \leq d(x)\leq 2\delta,%
\end{array}%
\right.
\end{equation*}%
belongs to $C^{1}(\overline{\Omega })\cap C_{0}(\Omega )$ and it is a
subsolution of the problem
\begin{equation}
\left\{
\begin{array}{ll}
-\Delta _{p(x)}u=\lambda u^{\gamma (x)} & \text{ in }\Omega , \\
u>0 & \text{ in }\Omega , \\
u=0 & \text{ on }\partial \Omega ,%
\end{array}%
\right.  \label{4}
\end{equation}%
for $\lambda \geq \lambda _{0}$ and $-1<\gamma ^{-}\leq \gamma ^{+}<0$.
According to definition of $w$, we have
\begin{equation}
\left\{
\begin{array}{ll}
w(x)=d(x) & \text{for }d(x)<\delta \\
\delta \leq w(x)\leq C_{\delta } & \text{for }d(x)\geq \delta ,%
\end{array}%
\right.  \label{w}
\end{equation}%
where $\delta ,C_{\delta }$ are positive constants independents of $\lambda $%
.

\begin{lemma}
\label{L6} Let $u$ the solution of (\ref{4}) given in \cite{QZ} for $\lambda
$ large enough. Then, for $\delta >0$ small enough, it holds
\begin{equation}
\min \{\delta ,d(x)\}\leq u(x)\leq C\lambda ^{\frac{1}{p^{-}-1}}\quad %
\mbox{in}\text{ }\Omega ,  \label{EQU1}
\end{equation}
where $C,\delta >0$ are constants independent of $\lambda $.
\end{lemma}

\begin{proof}
By using the fact that $w$ is a subsolution of (\ref{4}), Zhang in \cite{QZ}
showed that
\begin{equation*}
w(x)\leq u(x)\quad \text{for a.e. }x\in \Omega
\end{equation*}%
provided $\lambda $ large enough. Thus, it remains to prove that the second
inequality in (\ref{EQU1}) holds in $\Omega $. To this end, let a constant $%
k\geq 1$ and set $A_{k}=\{x\in \Omega :u(x)>k\}$. Taking $(u-k)^{+}$ as a
test function in (\ref{4}), we get
\begin{equation*}
\begin{array}{l}
\int_{A_{k}}\left\vert \nabla u\right\vert ^{p(x)}\text{ }dx=\lambda
\int_{A_{k}}w^{\gamma (x)}(u-k)\text{ }dx\leq \lambda \int_{A_{k}}(u-k)\text{
}dx.%
\end{array}%
\end{equation*}%
Then, following the quite similar argument as in \cite[Proof of Lemma 2.1]%
{Fan1} with $M=\lambda $ large, we obtain%
\begin{equation*}
\begin{array}{l}
u(x)\leq C\lambda ^{\frac{1}{p^{-}-1}}\text{ \ in }\Omega ,%
\end{array}%
\end{equation*}%
with a constant $C>0$ independent of $\lambda $, ending the proof of the
Lemma.
\end{proof}

The next result provides regularity of solutions for singular problems with
variable exponents. The constant case was proved by Hai in \cite{Hai} using
a different approach.

\begin{lemma}
\label{L1} Let $h:\Omega \rightarrow \mathbb{R}$ be a mensurable function
with
\begin{equation}
|h(x)|\leq Cd(x)^{-\gamma (x)}\quad \mbox{for}\quad x\in \Omega  \label{5}
\end{equation}%
where $\Omega \subset \mathbb{R}^{N}$ is a smooth bounded domain and $\gamma
:\overline{\Omega }\rightarrow \mathbb{R}$ is a continuous function such
that
\begin{equation}
\lim_{d(x)\rightarrow 0}N\gamma (x)=L\in (0,1).  \label{*}
\end{equation}%
If $u\in W_{0}^{1,p(x)}(\Omega )$ is a solution of the problem
\begin{equation}
\left\{
\begin{array}{ll}
-\Delta _{p(x)}u=h(x) & \text{in }\Omega \\
u=0 & \text{on }\partial \Omega ,%
\end{array}%
\right.  \label{**}
\end{equation}%
then there is a positive constant $M_{1},$ independent of $u,$ such that $%
|u|_{\infty }\leq M_{1}$. Moreover, $u\in C^{1,\alpha }(\overline{\Omega })$
and $\Vert u\Vert _{C^{1,\alpha }(\overline{\Omega })}\leq M_{1}$ with $%
\alpha \in (0,1),$ for some constant $M_{1}>0$ independent of $u$.
\end{lemma}

\begin{proof}
First, recall from \cite{LM} that for all $r\in \lbrack 0,1)$ we have
\begin{equation*}
\int_{\Omega }\frac{1}{d(x)^{r}}dx<\infty .
\end{equation*}%
Fixing $\epsilon >0$ such that $L+\epsilon \in (0,1)$, from (\ref{*}), we
derive
\begin{equation*}
\begin{array}{l}
\int_{\Omega }|h|^{N}dx<C\int_{\Omega }\frac{1}{d(x)^{N\gamma (x)}}dx\leq
C_{1}+C\int_{\Omega }\frac{1}{d(x)^{L+\epsilon }}dx<\infty ,%
\end{array}%
\end{equation*}%
for some constant $C_{1}>0$, showing that $h\in L^{N}(\Omega )$.

For each $k\in \mathbb{N}$, set
\begin{equation*}
A_{k}=\{x\in \Omega \,:\,u(x)>k\}.
\end{equation*}%
Since $u\in L^{1}(\Omega )$, we have that
\begin{equation*}
|A_{k}|\rightarrow 0\quad \mbox{as}\quad k\rightarrow +\infty .
\end{equation*}%
Once $h\in L^{N}(\Omega )$, it follows that%
\begin{equation*}
\int_{A_{k}}|h|^{N}dx\rightarrow 0\quad \mbox{as}\quad k\rightarrow +\infty
\end{equation*}%
or equivalently
\begin{equation}
|h|_{L^{N}(A_{k})}\rightarrow 0\quad \mbox{as}\quad k\rightarrow +\infty .
\label{E1}
\end{equation}%
Using $(u-k)^{+}$ as a test function in (\ref{**}), we get
\begin{equation*}
\int_{A_{k}}|\nabla u|^{p(x)}dx=\int_{A_{k}}h(u-k)^{+}dx\leq
|h|_{L^{N}(A_{k})}|(u-k)^{+}|_{L^{\frac{N}{N-1}}(\Omega )}
\end{equation*}%
Since $(u-k)^{+}\in W^{1,1}(\Omega )$, the Sobolev embedding leads to
\begin{equation*}
\int_{A_{k}}|\nabla u|^{p(x)}dx=\int_{A_{k}}h(u-k)^{+}dx\leq
C_{1}|h|_{L^{N}(A_{k})}\int_{A_{k}}|\nabla u|dx.
\end{equation*}%
From the estimate below
\begin{equation}
\int_{A_{k}}|\nabla u|dx\leq \int_{A_{k}}|\nabla u|^{p(x)}dx+|A_{k}|
\label{E2}
\end{equation}%
we derive that
\begin{equation*}
\int_{A_{k}}|\nabla u|^{p(x)}dx=\int_{A_{k}}h(u-k)^{+}dx\leq
C_{1}|h|_{L^{N}(A_{k})}\int {A_{k}}|\nabla
u|^{p(x)}dx+C_{1}|h|_{L^{N}(A_{k})}|A_{k}|.
\end{equation*}%
Thereby, for $k$ large enough the limit (\ref{E1}) gives
\begin{equation*}
\int_{A_{k}}|\nabla u|^{p(x)}dx=\int_{A_{k}}h(u-k)^{+}dx\leq C_{2}|A_{k}|.
\end{equation*}%
The last inequality together with (\ref{E2}) leads to
\begin{equation*}
\int_{A_{k}}|\nabla u|dx\leq C_{4}|A_{k}|.
\end{equation*}%
On the other hand, we know that
\begin{equation*}
\int_{A_{k}}(u-k)dx\leq |A_{k}|^{\frac{1}{N}}|(u-k)|_{L^{\frac{N}{N-1}%
}(A_{k})}\leq C_{3}|A_{k}|^{\frac{1}{N}}\int_{A_{k}}|\nabla u|dx,
\end{equation*}%
and so,
\begin{equation*}
\int_{A_{k}}(u-k)dx\leq C_{5}|A_{k}|^{1+\frac{1}{N}}.
\end{equation*}%
Then, owing to \cite[Lemma 5.1, Chaper 2]{LU} we conclude that there is $%
k_{1}>0,$ independent of $u,$ such that
\begin{equation}
u(x)\leq k_{1}\quad \mbox{a.e in}\quad \Omega .  \label{E3}
\end{equation}%
Now, observe that the function $-u$ verifies the problem
\begin{equation*}
\left\{
\begin{array}{ll}
-\Delta _{p(x)}(-u)=-h(x) & \text{in }\Omega \\
u=0 & \text{on }\partial \Omega .%
\end{array}%
\right.
\end{equation*}%
Then, repeating the same argument as above we get $k_{2}>0,$ independent of $%
u,$ such that
\begin{equation}
-u(x)\leq k_{2}\quad \mbox{a.e in}\quad \Omega .  \label{E4}
\end{equation}

From (\ref{E3}), (\ref{E4}), there is $M>0$ independent of $u$ such that
\begin{equation*}
|u(x)|\leq M\quad \mbox{a.e in}\quad \Omega ,
\end{equation*}%
from where it follows that $u\in L^{\infty }(\Omega )$ with
\begin{equation*}
|u|_{\infty }\leq M.
\end{equation*}%
Now, if $\psi \in C^{1,\alpha }(\overline{\Omega }),$ for certain $\alpha
\in (0,1),$ is a solution of the
\begin{equation*}
\left\{
\begin{array}{ll}
-\Delta \psi=h(x) & \text{in }\Omega \\
\psi=0 & \text{on }\partial \Omega ,%
\end{array}%
\right.
\end{equation*}%
we get
\begin{equation*}
-div(|\nabla u|^{p(x)-2}\nabla u-\nabla \psi)=0.
\end{equation*}%
Hence, the $C^{1,\alpha }$-boundedness of $u$ follows from \cite[Theorem 1.2]%
{Fan2}. This completes the proof.
\end{proof}

\begin{lemma}
\label{L2}Let $\varepsilon >0$ and $h,\tilde{h}\in L_{loc}^{\infty }(\Omega
) $ satisfy (\ref{5}) with $h\geq 0$, $h\neq 0$. Let $u,u_{\varepsilon }\in
W_{0}^{1,p(x)}(\Omega )$ be the solutions of problems
\begin{equation}
\left\{
\begin{array}{ll}
-\Delta _{p(x)}u=h(x) & \text{ in }\Omega , \\
u=0 & \text{ on }\partial \Omega ,%
\end{array}%
\right.  \label{12}
\end{equation}%
and
\begin{equation}
-\Delta _{p(x)}u_{\varepsilon }=\left\{
\begin{array}{ll}
h(x) & \text{ if \ }d(x)>\varepsilon \\
\tilde{h}(x) & \text{ if \ }d(x)<\varepsilon%
\end{array}%
\right. ,\text{ }u_{\varepsilon }=0\text{ \ on }\partial \Omega .  \label{13}
\end{equation}%
Then, for $\varepsilon $ small enough, it holds $u_{\varepsilon }\geq \frac{u%
}{2}$ in $\Omega $.
\end{lemma}

\begin{proof}
By Lemma \ref{L1} there exist $R>0$ and $\delta \in (0,1)$ such that
\begin{equation}
u,u_{\varepsilon }\in C^{1,\delta }(\overline{\Omega })\text{ \ and \ }%
\left\Vert u\right\Vert _{C^{1,\delta }(\overline{\Omega })},\left\Vert
u_{\varepsilon }\right\Vert _{C^{1,\delta }(\overline{\Omega })}\leq R.
\label{11}
\end{equation}%
Since $h\geq 0,$ the strong maximum principle (see \cite{FZZ}) implies
\begin{equation}
u(x)\geq cd(x)\text{ \ in }\Omega ,  \label{26}
\end{equation}%
for some constant $c>0$. Subtracting (\ref{12}) from (\ref{13}), multiplying
by $u-u_{\varepsilon }$ and integrating over $\Omega $ we obtain
\begin{equation*}
\begin{array}{l}
\int_{\Omega }(\left\vert \nabla u\right\vert ^{p(x)-2}\nabla u-\left\vert
\nabla u_{\varepsilon }\right\vert ^{p(x)-2}\nabla u_{\varepsilon })\nabla
(u-u_{\varepsilon })\ dx\leq 2R\int_{\{d(x)<\varepsilon \}}|h-\tilde{h}|%
\text{ }dx.%
\end{array}%
\end{equation*}%
Now let $\mathcal{K}_{1}=\{x\in \Omega :p(x)<2\}$ and $\mathcal{K}%
_{2}=\{x\in \Omega :p(x)\geq 2\}$. Then, using the algebraic inequalities
\begin{equation*}
\begin{array}{l}
|y_{1}-y_{2}|^{r}\leq \frac{1}{\gamma -1}%
[(|y_{1}|^{r-2}y_{1}-|y_{2}|^{r-2}y_{2})(y_{1}-y_{2})](|y_{1}|^{r}+|y_{2}|^{r})^{(2-r)/r}%
\text{,}%
\end{array}%
\end{equation*}
if $1<r<2$ and
\begin{equation*}
\begin{array}{l}
|y_{1}-y_{2}|^{r}\leq
2^{r}(|y_{1}|^{r-2}y_{1}-|y_{2}|^{r-2}y_{2})(y_{1}-y_{2})\text{ \ if }r\geq
2,%
\end{array}%
\end{equation*}%
for $y_{1},y_{2}\in
%TCIMACRO{\U{211d} }%
%BeginExpansion
\mathbb{R}
%EndExpansion
^{N},$ we obtain%
\begin{equation*}
\left\Vert \nabla (u-u_{\varepsilon })\right\Vert _{L^{p(x)}}\rightarrow 0\
\text{as}\ \varepsilon \rightarrow 0.
\end{equation*}%
Hence, by (\ref{11}) and the compact embedding $C^{1,\delta }(\overline{%
\Omega })\subset C^{1}(\overline{\Omega })$, we get $u\rightarrow
u_{\varepsilon }$ in $C^{1}(\overline{\Omega })$\ as \ $\varepsilon
\rightarrow 0.$ Consequently, from (\ref{26}) and for $\varepsilon $ small
enough, we have
\begin{equation*}
u-u_{\varepsilon }\leq \frac{c}{2}d\leq \frac{u}{2}\text{ \ in }\Omega ,
\end{equation*}%
which implies
\begin{equation*}
u_{\varepsilon }\geq u-\frac{u}{2}=\frac{u}{2}\ \ \text{in }\Omega .
\end{equation*}%
The proof is completed.
\end{proof}

\section{Sub-supersolution Theorems}

\label{sec3}

Let us introduce the quasilinear system
\begin{equation}
\left\{
\begin{array}{ll}
-\Delta _{p(x)}u=f(x,u,v) & \text{in }\Omega , \\
-\Delta _{q(x)}v=g(x,u,v) & \text{in }\Omega , \\
u,v>0 & \text{in }\Omega , \\
u,v=0 & \text{on }\partial \Omega ,%
\end{array}%
\right.  \tag{$P_{f,g}$}  \label{p*}
\end{equation}%
where $\Omega $ is a bounded domain in $%
%TCIMACRO{\U{211d} }%
%BeginExpansion
\mathbb{R}
%EndExpansion
^{N}$ $\left( N\geq 2\right) $ with smooth boundary and $f,g:\Omega \times
(0,+\infty )\times (0,+\infty )\rightarrow
%TCIMACRO{\U{211d} }%
%BeginExpansion
\mathbb{R}
%EndExpansion
$ are Carath\'{e}odory functions which can exhibit singularities when the
variables $u$ and $v$ approach zero. More precisely, for every $%
(s_{1},s_{2})\in
%TCIMACRO{\U{211d} }%
%BeginExpansion
\mathbb{R}
%EndExpansion
_{+}^{\ast }\times
%TCIMACRO{\U{211d} }%
%BeginExpansion
\mathbb{R}
%EndExpansion
_{+}^{\ast }$ and for almost every $x\in \Omega $, we assume that $f(\cdot
,s_{1},s_{2})$ and $g(\cdot ,s_{1},s_{2})$ are Lebesgue measurable in $%
\Omega $ and $f(x,\cdot ,\cdot )$ and $g(x,\cdot ,\cdot )$ are in $C(%
%TCIMACRO{\U{211d} }%
%BeginExpansion
\mathbb{R}
%EndExpansion
_{+}^{\ast }\times
%TCIMACRO{\U{211d} }%
%BeginExpansion
\mathbb{R}
%EndExpansion
_{+}^{\ast })$.

In what follows, we divide our study into two classes of systems, namely
cooperative system and competitive system.

\subsection{Cooperative System}

\label{subsec1}

The system (\ref{p*}) is called cooperative if for $u $ (resp. $v$) fixed
the nonlinearity $f$ (resp. $g$) is increasing in $v$ (resp. $u$).

We recall that a sub-supersolution for (\ref{p*}) is any pair $(\underline{u}%
,\underline{v})$, $(\overline{u},\overline{v})\in (W_{0}^{1,p(x)}(\Omega
)\cap L^{\infty }(\Omega ))\times (W_{0}^{1,q(x)}(\Omega )\cap L^{\infty
}(\Omega ))$ for which there hold $(\overline{u},\overline{v})\geq (%
\underline{u},\underline{v})$ in $\Omega $,
\begin{equation*}
\begin{array}{l}
\int_{\Omega }\left\vert \nabla \underline{u}\right\vert ^{p(x)-2}\nabla
\underline{u}\nabla \varphi \ dx-\int_{\Omega }f(x,\underline{u},\omega
_{2})\varphi \ dx\leq 0, \\
\int_{\Omega }\left\vert \nabla \underline{v}\right\vert ^{q(x)-2}\nabla
\underline{v}\nabla \psi \ dx-\int_{\Omega }g(x,\omega _{1},\underline{v}%
)\psi \ dx\leq 0,%
\end{array}%
\end{equation*}%
\begin{equation*}
\begin{array}{l}
\int_{\Omega }\left\vert \nabla \overline{u}\right\vert ^{p(x)-2}\nabla
\overline{u}\nabla \varphi \ dx-\int_{\Omega }f(x,\overline{u},\omega
_{2})\varphi \ dx\leq 0, \\
\int_{\Omega }\left\vert \nabla \overline{v}\right\vert ^{q(x)-2}\nabla
\overline{v}\nabla \psi -\int_{\Omega }g(x,\omega _{1},\overline{v})\psi \
dx\geq 0,%
\end{array}%
\end{equation*}%
for all $\left( \varphi ,\psi \right) \in W_{0}^{1,p(x)}\left( \Omega
\right) \times W_{0}^{1,q(x)}\left( \Omega \right) $ with $\varphi ,\psi
\geq 0$ a.e. in $\Omega $ and for all $\left( \omega _{1},\omega _{2}\right)
\in W_{0}^{1,p(x)}\left( \Omega \right) \times W_{0}^{1,q(x)}\left( \Omega
\right) $ satisfying $\omega _{1}\in \lbrack \underline{u},\overline{u}]$
and $\omega _{2}\in \lbrack \underline{v},\overline{v}]$ a.e. in $\Omega $ .

The main goal in this subsection is to prove Theorem \ref{T2} below, which
is a key point in the proof of Theorem \ref{T1}.

\begin{theorem}
\label{T2} Assume that system (\ref{p*}) is cooperative and let $\left(
\underline{u},\underline{v}\right) ,$ $\left( \overline{u},\overline{v}%
\right) \in C^{1}(\overline{\Omega })\times C^{1}(\overline{\Omega })$ be a
sub and supersolution pairs of (\ref{p*}). Suppose there exist constants $%
k_{1},k_{2}>0$ and $\alpha (x),\beta (x),$ with
\begin{equation}
\begin{array}{l}
-1\leq \alpha ^{-}\leq \alpha ^{+}<0,\text{ \ \ }-1\leq \beta ^{-}\leq \beta
^{+}<0%
\end{array}
\label{h6}
\end{equation}%
and%
\begin{equation}
\lim_{d(x)\rightarrow 0}N\alpha (x)=L_{1}\in (-1,0),\text{ \ \ \ }%
\lim_{d(x)\rightarrow 0}N\beta (x)=L_{2}\in (-1,0),  \label{h7}
\end{equation}%
such that
\begin{equation}
\begin{array}{c}
\left\vert f(x,u,v)\right\vert \leq k_{1}d(x)^{\alpha (x)}\text{ and }%
\left\vert g(x,u,v)\right\vert \leq k_{2}d(x)^{\beta (x)}\text{ in }\Omega
\times \lbrack \underline{u},\overline{u}]\times \lbrack \underline{v},%
\overline{v}].%
\end{array}
\label{h5}
\end{equation}%
Then, system (\ref{p*}) has a positive solution $(u,v)$ in $C^{1,\nu }(%
\overline{\Omega })\times C^{1,\nu }(\overline{\Omega })$ for certain $\nu
\in (0,1).$
\end{theorem}

\begin{proof}
For each $(z_{1},z_{2})\in C(\overline{\Omega })\times C(\overline{\Omega }
), $ let $(u,v)\in W_{0}^{1,p(x)}(\Omega )\times W_{0}^{1,q(x)}(\Omega )$ be
the unique solution of the problem
\begin{equation}
\left\{
\begin{array}{ll}
-\Delta _{p(x)}u=\widetilde{f}(x,z_{1},z_{2}) & \text{in }\Omega , \\
-\Delta _{q(x)}v=\widetilde{g}(x,z_{1},z_{2}) & \text{in }\Omega , \\
u,v>0 & \text{in }\Omega , \\
u,v=0 & \text{on }\partial \Omega ,%
\end{array}
\right.  \label{301}
\end{equation}
where
\begin{equation}
\widetilde{f}(x,z_{1},z_{2})=f(x,\widetilde{z}_{1},\widetilde{z}_{2})\text{
\ and \ }\widetilde{g}(x,z_{1},z_{2})=g(x,\widetilde{z}_{1},\widetilde{z}
_{2})  \label{320}
\end{equation}
with
\begin{equation}
\widetilde{z}_{1}=\min \left\{\max \left\{z_{1},\underline{u}\right\},%
\overline{u}\right\}\text{ and } \widetilde{z}_{2}=\min \left\{\max
\left\{z_{2},\underline{v}\right\},\overline{v}\right\}.  \label{300}
\end{equation}
Then $\underline{u}\leq \widetilde{z}_{1}\leq \overline{u}$ and $\underline{%
v }\leq \widetilde{z}_{2}\leq \overline{v}$ and by (\ref{h5}) we have
\begin{equation}
|\widetilde{f}(x,z_{1},z_{2})|\leq k_{1}d(x)^{\alpha (x)}\text{ \ and \ }
\left\vert \widetilde{g}(x,z_{1},z_{2})\right\vert \leq k_{2}d(x)^{\beta
(x)} \text{ for a.e. }x\in \Omega .  \label{302}
\end{equation}

Using the continuous embedding $W_{0}^{1,p(x)}(\Omega )\hookrightarrow
W_{0}^{1,p^{-}}(\Omega )$ together with (\ref{h6}), for each $\varphi \in
W_{0}^{1,p(x)}(\Omega )$ we have
\begin{equation*}
\begin{array}{l}
\int_{\Omega }|\varphi |d(x)^{\alpha (x)}\text{ }dx=\int_{\{d<1\}}|\varphi
|d(x)^{\alpha (x)}\text{ }dx+\int_{\{d\geq 1\}}|\varphi |d(x)^{\alpha (x)}%
\text{ }dx \\
\leq \int_{\{d<1\}}|\varphi |d(x)^{\alpha ^{+}}\text{ }dx+\int_{\{d\geq
1\}}|\varphi |\text{ }dx\leq C^{\prime }\left\Vert \varphi \right\Vert
_{W_{0}^{1,p^{-}}(\Omega )}<\infty ,%
\end{array}%
\end{equation*}%
for some constant $C^{\prime }>0$. Here, we used the Hardy-Sobolev
inequality which guarantees that $\varphi d(x)^{\alpha ^{+}}\in L^{r}(\Omega
)$ with $\frac{1}{r}=\frac{1}{p^{-}}-\frac{1+\alpha ^{+}}{N}$. In the same
manner, by using $W_{0}^{1,q(x)}(\Omega )\hookrightarrow
W_{0}^{1,q^{-}}(\Omega )$ and (\ref{h6}), for $\psi \in
W_{0}^{1,q(x)}(\Omega ),$ we can see that $\int_{\Omega }|\psi |d(x)^{\beta
(x)}$ $dx<\infty $. Hence, this ensures that
\begin{equation*}
\widetilde{f}(x,z_{1},z_{2})\in W^{-1,p^{\prime }(x)}(\Omega )\text{ and }%
\widetilde{g}(x,z_{1},z_{2})\in W^{-1,q^{\prime }(x)}(\Omega ),
\end{equation*}%
which in turns enable us to conclude, by Minty-Browder Theorem (see, e.g.,
\cite{B}),\ the uniqueness of the solution $(u,v)$ in (\ref{301}).

Let us introduce the operator
\begin{equation*}
\begin{array}{lll}
\mathcal{T}: & C(\overline{\Omega })\times C(\overline{\Omega }) &
\rightarrow C(\overline{\Omega })\times C(\overline{\Omega }) \\
& \text{ \ \ \ \ }(z_{1},z_{2}) & \mapsto \mathcal{T}(z_{1},z_{2})=(u,v).%
\end{array}%
\end{equation*}%
We will now prove, by applying Schauder's fixed point theorem, that $%
\mathcal{T}$ has a fixed point. Using (\ref{h7}) and Lemma \ref{L1}, there
exists $\nu \in (0,1)$ such that
\begin{equation}
(u,v)\in C^{1,\nu }(\overline{\Omega })\times C^{1,\nu }(\overline{\Omega })%
\text{ \ and \ }\left\Vert u\right\Vert _{C^{1,\nu }(\overline{\Omega }%
)},\left\Vert v\right\Vert _{C^{1,\nu }(\overline{\Omega })}\leq C,
\label{2}
\end{equation}%
where $C>0$ is independent of $u$ and $v$. Then the compactness of the
embedding $C^{1,\nu }(\overline{\Omega })\subset C(\overline{\Omega })$
implies that $\mathcal{T(}C(\overline{\Omega })\times C(\overline{\Omega }))$
is a relatively compact subset of $C(\overline{\Omega })\times C(\overline{%
\Omega })$.

Next, we show that $\mathcal{T}$ is continuous with respect to the topology
of $C(\overline{\Omega })\times C(\overline{\Omega })$. Let $%
(z_{1,n},z_{2,n})\rightarrow (z_{1},z_{2})$ in $C(\overline{\Omega })\times
C(\overline{\Omega })$ for all $n$. Denoting $\left( u_{n},v_{n}\right) =%
\mathcal{T(}z_{1,n},z_{2,n})$, we have from (\ref{2}) that $\left(
u_{n},v_{n}\right) \in C^{1,\nu }(\overline{\Omega })\times C^{1,\nu }(%
\overline{\Omega })$. By Ascoli-Arzel\`{a} theorem there holds $%
(u_{n},v_{n})\rightarrow (u,v)$ in $C(\overline{\Omega })\times C(\overline{%
\Omega })$. On the other hand, (\ref{h6}), (\ref{h5}) ensure that
\begin{equation*}
\widetilde{f}(x,z_{1,n},z_{2,n})\rightarrow \widetilde{f}(x,z_{1},z_{2})\in
W^{-1,p^{\prime }(x)}(\Omega )
\end{equation*}%
and
\begin{equation*}
\widetilde{g}(x,z_{1,n},z_{2,n})\rightarrow \widetilde{g}(x,z_{1},z_{2})\in
W^{-1,q^{\prime }(x)}(\Omega ).
\end{equation*}%
The above limits permit to conclude that $\mathcal{T}$ is continuous.

We are thus in a position to apply Schauder's fixed point theorem to the map
$\mathcal{T}$, which establishes the existence of $(u,v)\in C(\overline{
\Omega })\times C(\overline{\Omega })$ satisfying $(u,v)=\mathcal{T}(u,v).$

Let us justify that
\begin{equation*}
\underline{u}\leq u\leq \overline{u}\text{ and }\underline{v}\leq v\leq
\overline{v}\text{ in }\Omega .
\end{equation*}
Put $\zeta =(\underline{u}-u)^{+}$ and suppose $\zeta \neq 0$. Then, bearing
in mind that system (\ref{p*}) is cooperative, from (\ref{300}), (\ref{301})
and (\ref{320}), we infer that
\begin{equation*}
\begin{array}{c}
\int_{\{u<\underline{u}\}}|\nabla u|^{p(x)-2}\nabla u\nabla \zeta \
dx=\int_{\Omega }|\nabla u|^{p(x)-2}\nabla u\nabla \zeta \ dx=\int_{\{u<
\underline{u}\}}\widetilde{f}(x,u,v)\zeta \ dx \\
\\
=\int_{\{u<\underline{u}\}}f(x,\widetilde{u},\widetilde{v})\zeta \
dx=\int_{\{u<\underline{u}\}}f(x,\underline{u},\widetilde{v})\zeta \ dx\geq
\int_{\{u<\underline{u}\}}|\nabla \underline{u}|^{p(x)-2}\nabla \underline{u}
\nabla \zeta \ dx.%
\end{array}%
\end{equation*}
This implies that
\begin{equation*}
\begin{array}{c}
\int_{\{u<\underline{u}\}}(|\nabla \underline{u}|^{p(x)-2}\nabla \underline{u%
}-|\nabla {u} |^{p(x)-2}\nabla {u})\nabla \zeta \ dx\leq 0,%
\end{array}%
\end{equation*}
a contradiction. Hence $u\geq \underline{u}$ in $\Omega $. A quite similar
argument provides that $v\geq \underline{v}$ in $\Omega $. In the same way,
we prove that $u\leq \overline{u}$ and $v\leq \overline{v}$ in $\Omega $.

Finally, thanks to Lemma \ref{L1} one has $(u,v)\in C^{1,\nu }(\overline{%
\Omega })\times C^{1,\nu }(\overline{\Omega })$ for some $\nu \in (0,1)$.
This completes the proof.
\end{proof}

\subsection{Competitive system}

The system (\ref{p*}) is called a competitive system if for $u$ (resp. $v$)
fixed the nonlinearity $f$ (resp. $g$) is not increasing in $v$ (resp. $u$).
In sum, this is the complementary situation for system (\ref{p}) with
respect to the case considered in the subsection \ref{subsec1}.

\begin{theorem}
\label{T3} Assume that (\ref{p*}) is a competitive system with $f,g$ being $%
C^{1}$-function. Let $\left( u_{0},v_{0}\right) ,$ $\left(
u_{1},v_{1}\right) \in (W_{0}^{1,p(x)}(\Omega )\cap C(\overline{\Omega }%
))\times (W_{0}^{1,q(x)}(\Omega )\cap C(\overline{\Omega })),$ with $\left(
u_{1},v_{1}\right) \geq \left( u_{0},v_{0}\right) $ in $\Omega ,$ and
\begin{equation}
\left\{
\begin{array}{l}
\int_{\Omega }\left\vert \nabla u_{0}\right\vert ^{p(x)-2}\nabla u_{0}\nabla
\varphi \ dx-\int_{\Omega }f(x,u_{0},v_{0})\varphi \ dx\leq 0, \\
\int_{\Omega }\left\vert \nabla v_{0}\right\vert ^{q(x)-2}\nabla v_{0}\nabla
\psi \ dx-\int_{\Omega }g(x,u_{0},v_{0})\psi \ dx\leq 0,%
\end{array}%
\right.  \label{6**}
\end{equation}%
\begin{equation}
\left\{
\begin{array}{l}
\int_{\Omega }\left\vert \nabla u_{1}\right\vert ^{p(x)-2}\nabla u_{1}\nabla
\varphi \ dx-\int_{\Omega }f(x,u_{1},v_{1})\varphi \ dx\geq 0, \\
\int_{\Omega }\left\vert \nabla v_{1}\right\vert ^{q(x)-2}\nabla v_{1}\nabla
\psi -\int_{\Omega }g(x,u_{1},v_{1})\psi \ dx\geq 0,%
\end{array}%
\right.  \label{7**}
\end{equation}%
for all $\left( \varphi ,\psi \right) \in W_{0}^{1,p(x)}\left( \Omega
\right) \times W_{0}^{1,q(x)}\left( \Omega \right) $ with $\varphi ,\psi
\geq 0$ a.e. in $\Omega .$ Assume in addition that the following conditions
hold:

\begin{description}
\item[\textrm{(i)}] there exist constants $C_{0},C_{0}^{\prime }>0$ and
functions $\theta _{1}(x),\theta _{2}(x)\in C(\overline{\Omega }),$ with $%
\theta _{1}^{-},\theta _{2}^{-}>0,$ such that
\begin{equation}
\begin{array}{l}
u_{1}(x)\leq C_{0}d(x)^{\theta _{1}(x)}\text{ \ and \ }v_{1}(x)\leq
C_{0}^{\prime }d(x)^{\theta _{2}(x)}\text{ \ in }\Omega .%
\end{array}
\label{c1**}
\end{equation}

\item[\textrm{(ii)}] there exist constants $k_{1},k_{2}>0$ and functions $%
\alpha (x),\beta (x)\in C(\overline{\Omega })$ with
\begin{equation}
-1\leq \alpha ^{-}\leq \alpha ^{+}<0,\text{ \ \ }-1\leq \beta ^{-}\leq \beta
^{+}<0  \label{c2**}
\end{equation}
and%
\begin{equation}
\lim_{d(x)\rightarrow 0}N\alpha (x)=L_{1}\in (-1,0),\text{ \ \ \ }%
\lim_{d(x)\rightarrow 0}N\beta (x)=L_{2}\in (-1,0),  \label{c22}
\end{equation}
such that
\begin{equation}
\left\{
\begin{array}{c}
\left\vert f(x,u,v)\right\vert \leq k_{1}d(x)^{\alpha (x)} \\
\left\vert g(x,u,v)\right\vert \leq k_{2}d(x)^{\beta (x)}%
\end{array}%
\right. \text{, \ in }\Omega \times \lbrack u_{0},u_{1}]\times \lbrack
v_{0},v_{1}].  \label{c3**}
\end{equation}

\item[\textrm{(iii)}] there exist $C_{1},C_{1}^{\prime }>0$ and functions $%
\gamma _{1}(x),\gamma _{2}(x)\in C(\overline{\Omega })$ such that
\begin{equation}
\left\{
\begin{array}{c}
|\frac{\partial f}{\partial v}(x,u,v)|\leq C_{1}d(x)^{\gamma _{1}(x)} \\
|\frac{\partial g}{\partial u}(x,u,v)|\leq C_{1}^{\prime }d(x)^{\gamma
_{2}(x)}%
\end{array}%
\right. ,\text{ \ in }\Omega \times \lbrack u_{0},u_{1}]\times \lbrack
v_{0},v_{1}],  \label{c4**}
\end{equation}%
with%
\begin{equation}
\left\{
\begin{array}{l}
\gamma _{1}(x)+\theta _{2}(x)\geq -1 \\
\gamma _{2}(x)+\theta _{1}(x)\geq -1%
\end{array}%
\right. \text{ \ in }\Omega .  \label{c5**}
\end{equation}
\end{description}

Then system (\ref{p*}) has a positive solution $(u,v)$ in $C^{1,\nu }(%
\overline{\Omega })\times C^{1,\nu }(\overline{\Omega })$ for certain $\nu
\in (0,1).$
\end{theorem}

\begin{proof}
For $(z_{1},z_{2})\in C(\overline{\Omega })\times C(\overline{\Omega }),$
let $(u,v)\in W_{0}^{1,p(x)}(\Omega )\times W_{0}^{1,q(x)}(\Omega )$ be a
solution of the problem
\begin{equation}
\left\{
\begin{array}{ll}
L_{z_{2},p(x)}(u)=\widetilde{f}(x,z_{1},z_{2}) & \text{in }\Omega , \\
L_{z_{1},q(x)}(v)=\widetilde{g}(x,z_{1},z_{2}) & \text{in }\Omega , \\
u,v>0 & \text{in }\Omega , \\
u,v=0 & \text{on }\partial \Omega ,%
\end{array}
\right.  \label{3**}
\end{equation}%
where%
\begin{equation*}
\left\{
\begin{array}{l}
L_{z_{1},z_{2},p(x)}(u)=-\Delta _{p(x)}u+\rho \tilde{z}_{2}\max
\{d(x)^{\gamma _{1}(x)},|\tilde{z}_{1}|^{p(x)-2}\tilde{z}_{1},|u|^{p(x)-2}u\}
\\
L_{z_{1},z_{2},q(x)}(v)=-\Delta _{q(x)}v+\rho \tilde{z}_{1}\max
\{d(x)^{\gamma _{2}(x)},|\tilde{z}_{2}|^{q(x)-2}\tilde{z}_{2},|v|^{q(x)-2}v\}%
\end{array}
\right.
\end{equation*}%
and%
\begin{equation}
\left\{
\begin{array}{c}
\widetilde{f}(x,z_{1},z_{2})=f(x,\tilde{z}_{1},\tilde{z}_{2})+\rho \tilde{z}%
_{2}\max \{d(x)^{\gamma _{1}(x)},|\tilde{z}_{1}|^{p(x)-2}\tilde{z}_{1}\} \\
\widetilde{g}(x,z_{1},z_{2})=g(x,\tilde{z}_{1},\tilde{z}_{2})+\rho \tilde{z}%
_{1}\max \{d(x)^{\gamma _{2}(x)},|\tilde{z}_{2}|^{q(x)-2}\tilde{z}_{2}\},%
\end{array}%
\right.  \label{4**}
\end{equation}%
with
\begin{equation}
\tilde{z}_{1}=\min \left\{ \max \left\{ z_{1},u_{0}\right\} ,u_{1}\right\}
\text{ \ \ and \ \ }\tilde{z}_{2}=\min \left\{ \max \left\{
z_{2},v_{0}\right\} ,v_{1}\right\} .  \label{1**}
\end{equation}%
Obviously,
\begin{equation*}
u_{0}(x)\leq \tilde{z}_{1}(x)\leq u_{1}(x)\quad \mbox{and}\quad v_{0}(x)\leq
\tilde{z}_{2}(x)\leq v_{1}(x)\quad \mbox{in}\quad \Omega .
\end{equation*}%
In the sequel, we fix the constant $\rho >0$ in (\ref{4**}) sufficiently
large so that the following inequalities are satisfied:
\begin{equation*}
\begin{array}{c}
\frac{\partial f}{\partial s_{2}}(x,s_{1},s_{2})+\rho \max \{d(x)^{\gamma
_{1}(x)},|s_{1}|^{p(x)-2}s_{1}\}\geq 0%
\end{array}%
\end{equation*}%
and
\begin{equation*}
\begin{array}{c}
\frac{\partial g}{\partial s_{1}}(x,s_{1},s_{2})+\rho \max \{d(x)^{\gamma
_{2}(x)},|s_{2}|^{q(x)-2}s_{2}\}\geq 0,%
\end{array}%
\end{equation*}%
uniformly in $x\in \Omega ,$ for $(s_{1},s_{2})\in \lbrack
u_{0},u_{1}]\times \lbrack v_{0},v_{1}]$. By the above choice of $\rho $,
the term in the right-hand side of first (resp. second) equation in (\ref%
{3**}) increases as $v$ (resp. $u$) increases.

By (\ref{c3**}) and (\ref{1**}),
\begin{equation}
|f(x,\tilde{z}_{1},\tilde{z}_{2})|\leq k_{1}d(x)^{\alpha (x)}\text{ and }%
\left\vert g(x,\tilde{z}_{1},\tilde{z}_{2})\right\vert \leq k_{2}d(x)^{\beta
(x)}\text{ for a.e. }x\in \Omega .  \label{5**}
\end{equation}

Using continuous embedding $W_{0}^{1,p(x)}(\Omega )\hookrightarrow
W_{0}^{1,p^{-}}(\Omega )$ and (\ref{c2**}), for each $\varphi \in
W_{0}^{1,p(x)}(\Omega )$ we have
\begin{equation*}
\begin{array}{l}
\int_{\Omega }|\varphi |d(x)^{\alpha (x)}\text{ }dx=\int_{\{d<1\}}|\varphi
|d(x)^{\alpha (x)}\text{ }dx+\int_{\{d\geq 1\}}|\varphi |d(x)^{\alpha (x)}%
\text{ }dx \\
\leq \int_{\{d<1\}}|\varphi |d(x)^{\alpha ^{+}}\text{ }dx+\int_{\{d\geq
1\}}|\varphi |\text{ }dx\leq C^{\prime }\left\Vert \varphi \right\Vert
_{W_{0}^{1,p^{-}}(\Omega )}<\infty ,%
\end{array}%
\end{equation*}%
for some positive constant $C^{\prime }$. Here, we used the Hardy-Sobolev
inequality which guarantees that $\varphi d(x)^{\alpha ^{+}}\in L^{r}(\Omega
)$ with $\frac{1}{r}=\frac{1}{p^{-}}-\frac{1+\alpha ^{+}}{N}$. In the same
manner, by using $W_{0}^{1,q(x)}(\Omega )\hookrightarrow
W_{0}^{1,q^{-}}(\Omega )$ and (\ref{c2**}), for $\psi \in
W_{0}^{1,q(x)}(\Omega ),$ we can see that $\int_{\Omega }|\psi |d(x)^{\beta
(x)}$ $dx<\infty $. Furthermore, observe from (\ref{1**}) that
\begin{equation*}
\begin{array}{l}
d(x)^{\gamma _{1}(x)}\tilde{z}_{2}\leq d(x)^{\gamma _{1}(x)}{v_1}(x)\leq
C_{0}^{\prime }d(x)^{\gamma _{1}(x)+\theta _{2}(x)}\text{ \ for a.e. }x\in
\Omega%
\end{array}%
\end{equation*}%
and
\begin{equation*}
\begin{array}{l}
d(x)^{\gamma _{2}(x)}\tilde{z}_{1}\leq d(x)^{\gamma _{2}(x)}{u_1}(x)\leq
C_{0}d(x)^{\gamma _{2}(x)+\theta _{1}(x)}\text{ \ for a.e. }x\in \Omega .%
\end{array}%
\end{equation*}%
Thus, since $\gamma _{1}(x)+\theta _{2}(x)\geq -1$ and $\gamma
_{2}(x)+\theta _{1}(x)\geq -1$ in $\Omega $ (see (\ref{c5**})), similar to
the above argument implies that
\begin{equation*}
\begin{array}{l}
\int_{\Omega }|\varphi |d(x)^{\gamma _{1}(x)+\theta _{2}(x)}dx,\text{ \ }%
\int_{\Omega }|\psi |d(x)^{\gamma _{2}(x)+\theta _{1}(x)}dx<\infty ,%
\end{array}%
\end{equation*}%
for all $(\varphi ,\psi )\in W_{0}^{1,p(x)}(\Omega )\times
W_{0}^{1,q(x)}(\Omega )$. Then, we deduce that
\begin{equation*}
\widetilde{f}(x,z_{1},z_{2})\in W^{-1,p^{\prime }(x)}(\Omega )\text{ and }%
\widetilde{g}(x,z_{1},z_{2})\in W^{-1,q^{\prime }(x)}(\Omega ),
\end{equation*}%
which in turns enable us to conclude, by Minty-Browder Theorem (see, e.g.,
\cite{B}),\ the uniqueness of the solution $(u,v)$ in (\ref{3**}).

Let us introduce the operator
\begin{equation*}
\begin{array}{lll}
\mathcal{T}: & C(\overline{\Omega })\times C(\overline{\Omega }) &
\rightarrow C(\overline{\Omega })\times C(\overline{\Omega }) \\
& \text{ \ \ \ \ }(z_{1},z_{2}) & \mapsto \mathcal{T}(z_{1},z_{2})=(u,v).%
\end{array}%
\end{equation*}%
and let prove, applying Schauder's fixed point theorem, that $\mathcal{T}$
has a fixed point. Observe from (\ref{1**}) that
\begin{equation*}
\begin{array}{l}
\max \{d(x)^{\gamma _{1}(x)},|\tilde{z}_{1}|^{p(x)-2}\tilde{z}%
_{1},|u|^{p(x)-2}u\}-\max \{d(x)^{\gamma _{1}(x)},|\tilde{z}_{1}|^{p(x)-2}%
\tilde{z}_{1}\}\geq 0\text{ \ in }\Omega%
\end{array}%
\end{equation*}%
and%
\begin{equation*}
\begin{array}{l}
\max \{d(x)^{\gamma _{2}(x)},|\tilde{z}_{2}|^{p(x)-2}\tilde{z}%
_{2},|v|^{q(x)-2}v\}-\max \{d(x)^{\gamma _{2}(x)},|\tilde{z}_{2}|^{q(x)-2}%
\tilde{z}_{2}\}\geq 0\text{ \ in }\Omega .%
\end{array}%
\end{equation*}%
Then, by (\ref{c3**}), one has
\begin{equation*}
-\Delta _{p(x)}u\leq f(x,\tilde{z}_{1},\tilde{z}_{2})\leq k_{1}d(x)^{\alpha
(x)}\text{ in }\Omega
\end{equation*}%
and%
\begin{equation*}
-\Delta _{q(x)}v\leq g(x,\tilde{z}_{1},\tilde{z}_{2})\leq k_{2}d(x)^{\beta
(x)}\text{ in }\Omega .
\end{equation*}%
Hence, using (\ref{c22}), Lemma \ref{L1} guarantees that there exist a
constant $C>0$ and $\nu \in (0,1)$ such that
\begin{equation}
(u,v)\in C^{1,\nu }(\overline{\Omega })\times C^{1,\nu }(\overline{\Omega })%
\text{ \ and \ }\left\Vert u\right\Vert _{C^{1,\nu }(\overline{\Omega }%
)},\left\Vert v\right\Vert _{C^{1,\nu }(\overline{\Omega })}\leq C,
\label{2**}
\end{equation}%
where $C>0$ is independent of $u$ and $v$. Then the compactness of the
embedding $C^{1,\nu }(\overline{\Omega })\subset C(\overline{\Omega })$
implies that $\mathcal{T}$ is continuous and compact operator with respect
to the topology of $C(\overline{\Omega })\times C(\overline{\Omega })$.

We are thus in a position to apply Schauder's fixed point theorem to the map
$\mathcal{T}$, which establishes the existence of $(u,v)\in C(\overline{
\Omega })\times C(\overline{\Omega })$ satisfying $(u,v)=\mathcal{T}(u,v).$

Let us justify that%
\begin{equation*}
u_{0}\leq u\leq u_{1}\text{ and }v_{0}\leq v\leq v_{1}\text{ in }\Omega .
\end{equation*}%
Put $w_{1}=(u_{0}-u)^{+},$ $w_{2}=(v_{0}-v)^{+}$. From (\ref{4**}), (\ref%
{1**}) and (\ref{6**}),
\begin{equation*}
\begin{array}{l}
\int_{\{u<u_{0}\}}|\nabla u|^{p(x)-2}\nabla u\nabla w_{1}\ dx+\rho
\int_{\{u<u_{0}\}}\tilde{v}\max \{d(x)^{\gamma _{1}(x)},|\tilde{u}|^{p(x)-2}%
\tilde{u},|u|^{p(x)-2}u\}w_{1}\text{ }dx \\
=\int_{\Omega }|\nabla u|^{p(x)-2}\nabla u\nabla w_{1}\ dx+\rho \int_{\Omega
}\tilde{v}\max \{d(x)^{\gamma _{1}(x)},|\tilde{u}|^{p(x)-2}\tilde{u}%
,|u|^{p(x)-2}u\}w_{1}\text{ }dx \\
=\int_{\{u<u_{0}\}}\widetilde{f}(x,u,v)w_{1}\ dx \\
=\int_{\{u<u_{0}\}}f(x,\tilde{u},\tilde{v})w_{1}\ dx+\rho \int_{\{u<u_{0}\}}%
\tilde{v}\max \{d(x)^{\gamma _{1}(x)},|\tilde{u}|^{p(x)-2}\tilde{u}\}w_{1}\
dx \\
=\int_{\{u<u_{0}\}}f(x,u_{0},\tilde{v})w_{1}\ dx+\rho \int_{\{u<u_{0}\}}%
\tilde{v}\max \{d(x)^{\gamma _{1}(x)},|u_{0}|^{p(x)-2}u_{0}\}w_{1}\ dx \\
\geq \int_{\{u<u_{0}\}}f(x,u_{0},v_{0})w_{1}\ dx+\rho \int_{\{u<{u_{0}}%
\}}v_{0}\max \{d(x)^{\gamma _{1}(x)},|u_{0}|^{p(x)-2}u_{0}\}w_{1}\ dx \\
\geq \int_{\{u<u_{0}\}}|\nabla u_{0}|^{p(x)-2}\nabla u_{0}\nabla w_{1}\
dx+\rho \int_{\{u<u_{0}\}}v_{0}\max \{d(x)^{\gamma
_{1}(x)},|u_{0}|^{p(x)-2}u_{0}\}w_{1}\ dx%
\end{array}%
\end{equation*}%
and similarly%
\begin{equation*}
\begin{array}{l}
\int_{\{v<v_{0}\}}|\nabla v|^{q(x)-2}\nabla v\nabla w_{2}\ dx+\rho
\int_{\{v<v_{0}\}}\tilde{u}\{d(x)^{\gamma _{2}(x)},|\tilde{v}|^{q(x)-2}%
\tilde{v},|v|^{q(x)-2}v\}w_{2}\text{ }dx \\
\geq \int_{\{v<v_{0}\}}|\nabla v_{0}|^{q(x)-2}\nabla v_{0}\nabla w_{2}\
dx+\rho \int_{\{v<v_{0}\}}u_{0}\max \{d(x)^{\gamma
_{2}(x)},|v_{0}|^{q(x)-2}v_{0}\}w_{2}\ dx.%
\end{array}%
\end{equation*}%
This implies that
\begin{equation*}
\begin{array}{l}
\int_{\{u<u_{0}\}}(|\nabla u_{0}|^{p(x)-2}\nabla u_{0}-|\nabla
u|^{p(x)-2}\nabla u)\nabla w_{1}\ dx \\
+\rho \int_{\{u<u_{0}\}}(v_{0}\max \{d(x)^{\gamma
_{1}(x)},|u_{0}|^{p(x)-2}u_{0}\}-\tilde{v}\max \{d(x)^{\gamma _{1}(x)},|%
\tilde{u}|^{p(x)-2}\tilde{u},|u|^{p(x)-2}u\})w_{1}\text{ }dx\leq 0%
\end{array}%
\end{equation*}%
and%
\begin{equation*}
\begin{array}{l}
\int_{\{v<v_{0}\}}(|\nabla v_{0}|^{q(x)-2}\nabla v_{0}-|\nabla
v|^{q(x)-2}\nabla v)\nabla w_{2}\ dx \\
+\rho \int_{\{v<v_{0}\}}(u_{0}\max \{d(x)^{\gamma
_{2}(x)},|v_{0}|^{q(x)-2}v_{0}\}-\tilde{u}\{d(x)^{\gamma _{2}(x)},|\tilde{v}%
|^{q(x)-2}\tilde{v},|v|^{q(x)-2}v\})w_{2}\text{ }dx\leq 0,%
\end{array}%
\end{equation*}%
showing that $u\geq u_{0}$ and $v\geq v_{0}$ in $\Omega $. A quite similar
argument provides that $u\leq u_{1}$ and $v\leq v_{1}$ in $\Omega $.

Finally, thanks to Lemma \ref{L1} one has $(u,v)\in C^{1,\nu }(\overline{%
\Omega })\times C^{1,\nu }(\overline{\Omega })$ for some $\nu \in (0,1)$.
This completes the proof.
\end{proof}

By strengthening the hypotheses on functions $\gamma _{1}$ and $\gamma _{2}$%
, the conclusion in Theorem \ref{T3} is still true if we drop the assumption
\textrm{(i)} by\textrm{\ }assuming that\textrm{\ }$\left( u_{1},v_{1}\right)
$ don't behaves as function $d(x)$ in $\Omega $. This is stated in the next
result which is a variant of Theorem \ref{T3}.

\begin{theorem}
\label{T4} Let $f,g,\alpha $ and $\beta $ as in Theorem \ref{T3} and assume $%
\left( u_{0},v_{0}\right) ,$ $\left( u_{1},v_{1}\right) \in
(W_{0}^{1,p(x)}(\Omega )\cap C(\overline{\Omega }))\times
(W_{0}^{1,q(x)}(\Omega )\cap C(\overline{\Omega })),$ with $\left(
u_{1},v_{1}\right) \geq \left( u_{0},v_{0}\right) $ in $\Omega ,$ satisfy (%
\ref{6**}) and (\ref{7**}). Suppose that \textrm{(iii) }holds with%
\begin{equation}
0>\gamma _{i}^{+}\geq \gamma _{i}^{-}\geq -1,\quad \mbox{for}\quad i=1,2.
\end{equation}%
Then system (\ref{p*}) has a positive solution $(u,v)$ in $C_{0}^{1,\nu }(%
\overline{\Omega })\times C_{0}^{1,\nu }(\overline{\Omega })$ for certain $%
\nu \in (0,1).$
\end{theorem}

\begin{proof}
From (\ref{1**}), notice that
\begin{equation*}
\begin{array}{l}
d(x)^{\gamma _{1}(x)}\tilde{z}_{2}\leq d(x)^{\gamma _{1}(x)}{v_{1}}(x)\leq
C_{0}^{\prime }d(x)^{\gamma _{1}(x)}\text{ \ for a.e. }x\in \Omega%
\end{array}%
\end{equation*}%
and
\begin{equation*}
\begin{array}{l}
d(x)^{\gamma _{2}(x)}\tilde{z}_{1}\leq d(x)^{\gamma _{2}(x)}{u_{1}}(x)\leq
C_{0}d(x)^{\gamma _{2}(x)}\text{ \ for a.e. }x\in \Omega .%
\end{array}%
\end{equation*}%
Then, the proof can be achieved by following a quite similar argument in
Theorem \ref{T3}.
\end{proof}

\section{Proof of Theorem \protect\ref{T1}}

\label{sec4}

Given a constant $\sigma >0$, let $w_{1}$ and $w_{2}$ be solutions of the
homogeneous Dirichlet problems
\begin{equation}
\left\{
\begin{array}{ll}
-\Delta _{p(x)}w_{1}=\lambda ^{\sigma }w_{1}^{\alpha _{1}(x)} & \text{ in }%
\Omega \\
w_{1}>0 & \text{ in }\Omega \\
w_{1}=0 & \text{ on }\partial \Omega ,%
\end{array}%
\right. ,\text{ \ }\left\{
\begin{array}{ll}
-\Delta _{q(x)}w_{2}=\lambda ^{\sigma }w_{2}^{\beta _{2}(x)} & \text{ in }%
\Omega \\
w_{2}>0 & \text{ in }\Omega \\
w_{2}=0 & \text{ on }\partial \Omega ,%
\end{array}%
\right.  \label{10}
\end{equation}%
which are known to satisfy%
\begin{equation}
\min \{\delta ,d(x)\}\leq w_{1}(x)\leq C_{1}\lambda ^{\frac{\sigma }{p^{-}-1}%
}\quad \mbox{in}\quad \Omega  \label{66}
\end{equation}%
and%
\begin{equation}
\min \{\delta ,d(x)\}\leq w_{2}(x)\leq C_{2}\lambda ^{\frac{\sigma }{q^{-}-1}%
}\quad \mbox{in}\quad \Omega .  \label{67}
\end{equation}
for some positive constant $C_{1},$ $C_{2}$ independent of $\lambda $ and
for $\delta >0$ small (see Lemma \ref{L6}).

Fix $\sigma \in (0,1)$ and let consider the functions $\underline{u}$ and $%
\underline{v}$ defined by
\begin{equation}
-\Delta _{p(x)}\underline{u}=\left\{
\begin{array}{ll}
\lambda ^{\sigma }w_{1}^{\alpha _{1}(x)} & \text{ in \ }\Omega \backslash
\overline{\Omega }_{\delta } \\
-w_{1}^{\alpha _{1}(x)} & \text{\ in \ }\Omega _{\delta }%
\end{array}%
\right. ,\text{ }\underline{u}=0\text{ \ on }\partial \Omega  \label{30}
\end{equation}%
and
\begin{equation}
-\Delta _{q(x)}\underline{v}=\left\{
\begin{array}{ll}
\lambda ^{\sigma }w_{2}^{\beta _{2}(x)} & \text{\ in \ }\Omega \backslash
\overline{\Omega }_{\delta } \\
-w_{2}^{\beta _{2}(x)} & \text{\ in \ }\Omega _{\delta }%
\end{array}%
\right. ,\text{ }\underline{v}=0\text{ \ on }\partial \Omega ,  \label{30*}
\end{equation}%
where
\begin{equation}
\Omega _{\delta }=\left\{ x\in \Omega :d\left( x,\partial \Omega \right)
<\delta \right\} ,  \label{70}
\end{equation}%
with a constant $\delta >0$ small. Using $W_{0}^{1,p(x)}(\Omega
)\hookrightarrow W_{0}^{1,p^{-}}(\Omega )$ together with (\ref{h1}) and (\ref%
{h3}), for each $\varphi \in W_{0}^{1,p(x)}(\Omega )$ we get
\begin{equation*}
\begin{array}{l}
\int_{\Omega }|\varphi |d(x)^{\alpha _{1}(x)}\text{ }dx=\int_{\{d<1\}}|%
\varphi |d(x)^{\alpha _{1}(x)}\text{ }dx+\int_{\{d\geq 1\}}|\varphi
|d(x)^{\alpha _{1}(x)}\text{ }dx \\
\leq \int_{\{d<1\}}|\varphi |d(x)^{\alpha _{1}^{+}}\text{ }dx+\int_{\{d\geq
1\}}|\varphi |\text{ }dx\leq C^{\prime }\left\Vert \varphi \right\Vert
_{W_{0}^{1,p^{-}}(\Omega )}%
\end{array}%
\end{equation*}%
for some positive constant $C^{\prime }$. Here we used the Hardy-Sobolev
Inequality which guarantees that $\varphi d(x)^{\alpha _{1}^{+}}\in
L^{r}(\Omega )$ with $\frac{1}{r}=\frac{1}{p^{-}}-\frac{1+\alpha _{1}^{+}}{N}
$. Similar arguments furnishes that there is $C^{\prime }>0$ such that
\begin{equation*}
\int_{\Omega }|\psi |d(x)^{\beta _{2}(x)}dx\leq C^{\prime }\left\Vert \psi
\right\Vert _{W_{0}^{1,q^{-}}(\Omega )},\quad \forall \psi \in
W_{0}^{1,q(x)}(\Omega ).
\end{equation*}%
Hence, the right-hand side of (\ref{30}) and (\ref{30*}) belongs to $%
W^{-1,p^{\prime }(x)}(\Omega )$ and $W^{-1,q^{\prime }(x)}(\Omega )$,
respectively. Consequently, the Minty-Browder Theorem (see \cite[Theorem V.15%
]{B}) implies the existence and uniqueness of $\underline{u}$ and $%
\underline{v}$ in (\ref{30}) and (\ref{30*}). Moreover, (\ref{10}), (\ref{30}%
), (\ref{30*}) and Lemma \ref{L2} together with the weak comparison
principle yield
\begin{equation}
\begin{array}{l}
\frac{w_{1}(x)}{2}\leq \underline{u}(x)\leq w_{1}(x)\text{ \ \ in }\Omega ,%
\end{array}
\label{31}
\end{equation}%
and
\begin{equation}
\begin{array}{l}
\frac{w_{2}(x)}{2}\leq \underline{v}(x)\leq w_{2}(x)\text{ \ \ in }\Omega .%
\end{array}
\label{31*}
\end{equation}

In what follows, we fix $\tilde{\Omega}$ as a smooth bounded domain in $%
%TCIMACRO{\U{211d} }%
%BeginExpansion
\mathbb{R}
%EndExpansion
^{N}$ such that $\overline{\Omega }\subset \tilde{\Omega}.$ Denote by $%
\tilde{d}(x)=dist(x,\partial \tilde{\Omega}).$ Define $\overline{u}$ and $%
\overline{v}$ in $C^{1,\nu }(\overline{\tilde{\Omega}}),$ for certain $\nu
\in (0,1)$, as the unique weak solutions of the problems
\begin{equation}
\left\{
\begin{array}{ll}
-\Delta _{p(x)}\overline{u}=\lambda ^{\bar{\sigma}} & \text{in }\tilde{\Omega%
}, \\
\overline{u}=0 & \text{on }\partial \tilde{\Omega},%
\end{array}%
\right. ,\text{ \ }\left\{
\begin{array}{ll}
-\Delta _{q(x)}\overline{v}=\lambda ^{\bar{\sigma}} & \text{in }\tilde{\Omega%
}, \\
\overline{v}=0 & \text{on }\partial \tilde{\Omega},%
\end{array}%
\right.   \label{20}
\end{equation}%
where the constant $\bar{\sigma}>0$ verifies%
\begin{equation}
\begin{array}{l}
\bar{\sigma}>\max \{\frac{p^{-}-1}{p^{-}-1-\alpha _{2}^{+}},\frac{q^{-}-1}{%
q^{-}-1-\beta _{1}^{+}}\}.%
\end{array}
\label{3}
\end{equation}%
It is known that $\overline{u}$ and $\overline{v}$ satisfy
\begin{equation}
\begin{array}{l}
\overline{u}(x)\leq c_{2}\lambda ^{\frac{\bar{\sigma}}{p^{-}-1}}\text{ \ and
\ }\overline{v}(x)\leq c_{2}^{\prime }\lambda ^{\frac{\bar{\sigma}}{q^{-}-1}}%
\text{ \ in }\tilde{\Omega},%
\end{array}
\label{21}
\end{equation}%
and
\begin{equation}
c_{0}\delta \leq \min \{\overline{u}(x),\overline{v}(x)\}\quad \mbox{if}%
\quad \tilde{d}(x)\geq \delta ,  \label{211}
\end{equation}%
where $c_{0}$ is independent of $\lambda $ large enough (see \cite{YY}).
From this, we have that
\begin{equation}
\begin{array}{l}
c_{0}\delta \leq \overline{u}(x)\leq c_{2}\lambda ^{\frac{\bar{\sigma}}{%
p^{-}-1}}\text{ \ and \ }c_{0}\delta \leq \overline{v}(x)\leq c_{2}^{\prime
}\lambda ^{\frac{\bar{\sigma}}{q^{-}-1}}\text{ \ in }\,\,\overline{\Omega },%
\end{array}
\label{2111}
\end{equation}%
for $\delta >0$ sufficiently small.

\begin{lemma}
\label{L3}Under assumptions (\ref{h3}) and (\ref{h1}), for $\lambda >0$
sufficiently large, $(\underline{u},\underline{v})$ and $(\overline{u},%
\overline{v})$ are subsolution and supersolution for problem (\ref{p})
respectively.
\end{lemma}

\begin{proof}
First of all, $\left( \overline{u},\overline{v}\right) \geq \left(
\underline{u},\underline{v}\right) $ in $\overline{\Omega },$ for $\lambda $
sufficiently large. Indeed, From (\ref{20}), (\ref{66}), (\ref{67}), (\ref%
{30}), (\ref{30*}) and since $0<\sigma <1<\bar{\sigma}$, one has%
\begin{equation*}
\begin{array}{l}
-\Delta _{p(x)}\overline{u}=\lambda ^{\bar{\sigma}}\geq -w_{1}^{\alpha
_{1}(x)}=-\Delta _{p(x)}\underline{u}\text{ \ in }\Omega _{\delta },%
\end{array}%
\end{equation*}%
\begin{equation*}
\begin{array}{l}
-\Delta _{q(x)}\overline{v}=\lambda ^{\bar{\sigma}}\geq -w_{2}^{\beta
_{2}(x)}=-\Delta _{q(x)}\underline{v}\text{ \ in }\Omega _{\delta },%
\end{array}%
\end{equation*}%
\begin{equation*}
\begin{array}{l}
-\Delta _{p(x)}\overline{u}=\lambda ^{\bar{\sigma}}\geq \lambda ^{\sigma
}\delta ^{\alpha _{1}(x)}\geq \lambda ^{\sigma }w_{1}^{\alpha
_{1}(x)}=-\Delta _{p(x)}\underline{u}\text{ \ in }\Omega \backslash
\overline{\Omega }_{\delta }%
\end{array}%
\end{equation*}%
and
\begin{equation*}
\begin{array}{l}
-\Delta _{q(x)}\overline{v}=\lambda ^{\bar{\sigma}}\geq \lambda ^{\sigma
}\delta ^{\beta _{2}(x)}\geq \lambda ^{\sigma }w_{2}^{\beta _{2}(x)}=-\Delta
_{q(x)}\underline{v}\text{ \ \ in }\Omega \backslash \overline{\Omega }%
_{\delta },%
\end{array}%
\end{equation*}%
provided that $\lambda $ is large enough. Then the monotonicity of the
operators $-\Delta _{p(x)}$ and $-\Delta _{q(x)}$ lead to
\begin{equation}
\underline{u}\leq \overline{u}\quad \mbox{and}\quad \underline{v}\leq
\overline{v}\quad \mbox{in}\quad \Omega ,  \label{EQ1}
\end{equation}%
for $\lambda $ sufficiently large.

Now, we will show that $\left( \overline{u},\overline{v}\right) $ is a
subsolution for (\ref{p}). In fact, by (\ref{10}), (\ref{30}) and (\ref{30*}%
), we have%
\begin{equation}
\begin{array}{l}
-\underline{u}^{-\alpha _{1}(x)}\underline{v}^{-\beta _{1}(x)}w_{1}^{\alpha
_{1}(x)}\leq 0\leq \lambda \text{ \ in }\Omega _{\delta }%
\end{array}
\label{15}
\end{equation}%
and
\begin{equation}
\begin{array}{l}
-\underline{u}^{-\alpha _{2}(x)}\underline{v}^{-\beta _{2}(x)}w_{2}^{\beta
_{2}(x)}\leq 0\leq \lambda \text{ \ in }\Omega _{\delta },%
\end{array}
\label{16}
\end{equation}%
for all $\lambda >0$. On the other hand, from (\ref{h3}), (\ref{h1}), (\ref%
{31}) and (\ref{31*}), since $\sigma \in (0,1)$, we obtain
\begin{equation}
\begin{array}{l}
\lambda ^{\sigma }\underline{u}^{-\alpha _{1}(x)}\underline{v}^{-\beta
_{1}(x)}w_{1}^{\alpha _{1}(x)}\leq \lambda ^{\sigma }w_{1}^{-\alpha _{1}(x)}(%
\frac{w_{2}}{2})^{-\beta _{1}(x)}w_{1}^{\alpha _{1}(x)} \\
\leq \lambda ^{\sigma }\delta ^{-\beta _{1}(x)}\leq \lambda \text{ in }%
\Omega \backslash \overline{\Omega }_{\delta },%
\end{array}
\label{17}
\end{equation}%
and
\begin{equation}
\begin{array}{l}
\lambda ^{\sigma }\underline{u}^{-\alpha _{2}(x)}\underline{v}^{-\beta
_{2}(x)}w_{2}^{\beta _{2}(x)}\leq \lambda ^{\sigma }(\frac{w_{1}}{2}%
)^{-\alpha _{2}(x)}w_{2}^{-\beta _{2}(x)}w_{2}^{\beta _{2}(x)} \\
\leq \lambda ^{\sigma }\delta ^{-\alpha _{2}(x)}\leq \lambda \text{ \ in }%
\Omega \backslash \overline{\Omega }_{\delta },%
\end{array}
\label{18}
\end{equation}%
provided that $\lambda $ is sufficiently large. Let $\left( \varphi ,\psi
\right) \in W_{0}^{1,p(x)}\left( \Omega \right) \times W_{0}^{1,q(x)}\left(
\Omega \right) $ with $\varphi ,\psi \geq 0$ a.e. in $\Omega $. Using (\ref%
{15})-(\ref{18}), (\ref{30}) and (\ref{30*}), it follows that
\begin{equation*}
\begin{array}{l}
\int_{\Omega }\left\vert \nabla \underline{u}\right\vert ^{p(x)-2}\nabla
\underline{u}\nabla \varphi \ dx=\lambda ^{\sigma }\int_{\Omega \backslash
\overline{\Omega }_{\delta }}w_{1}^{\alpha _{1}(x)}\varphi \text{ }%
dx-\int_{\Omega _{\delta }}w_{1}^{\alpha _{1}(x)}\varphi \text{ }dx \\
\leq \lambda \int_{\Omega }\underline{u}^{\alpha _{1}(x)}\underline{v}%
^{\beta _{1}(x)}\varphi \ dx\leq \lambda \int_{\Omega }\underline{u}^{\alpha
_{1}(x)}w^{\beta _{1}(x)}\varphi \text{ }dx%
\end{array}%
\end{equation*}%
and
\begin{equation*}
\begin{array}{l}
\int_{\Omega }\left\vert \nabla \underline{v}\right\vert ^{q(x)-2}\nabla
\underline{v}\nabla \psi \ dx=\lambda ^{\sigma }\int_{\Omega \backslash
\overline{\Omega }_{\delta }}w_{2}^{\beta _{2}(x)}\psi \text{ }%
dx-\int_{\Omega _{\delta }}w_{2}^{\beta _{2}(x)}\psi \text{ }dx \\
\leq \lambda \int_{\Omega }\underline{u}^{\alpha _{2}(x)}\underline{v}%
^{\beta _{2}(x)}\psi \ dx\leq \lambda \int_{\Omega }\zeta ^{\alpha _{2}(x)}%
\underline{v}^{\beta _{2}(x)}\psi \ dx%
\end{array}%
\end{equation*}%
for $\lambda >0$ sufficiently large, $\zeta \in \lbrack \underline{u},%
\overline{u}]$, $w\in \lbrack \underline{v},\overline{v}]$ and $\left(
\varphi ,\psi \right) \in W_{0}^{1,p(x)}\left( \Omega \right) \times
W_{0}^{1,q(x)}\left( \Omega \right) $ with $\varphi ,\psi \geq 0$ a.e. in $%
\Omega $. This shows that $(\underline{u},\underline{v})$ is a subsolution
for problem (\ref{p}).

The task is now to prove that $\left( \overline{u},\overline{v}\right) $
defined in (\ref{20}) is a supersolution of (\ref{p}). On account of (\ref%
{h1}), (\ref{h3}), (\ref{20}), (\ref{3}) and (\ref{2111}), one has
\begin{equation*}
\begin{array}{l}
-\Delta _{p(x)}\overline{u}=\lambda ^{\bar{\sigma}}\geq \lambda ^{1+\frac{%
\bar{\sigma}\beta _{1}^{+}}{q^{-}-1}}(c_{0}\delta )^{\alpha
_{1}^{-}}(c_{1}^{\prime })^{\beta _{1}^{+}} \\
\geq \lambda ^{1+\frac{\bar{\sigma}\beta _{1}(x)}{q^{-}-1}}(c_{0}\delta
)^{\alpha _{1}(x)}(c_{2}^{\prime })^{\beta _{1}(x)}\geq \lambda \overline{u}%
^{\alpha _{1}(x)}\overline{v}^{\beta _{1}(x)}\text{ \ in }\overline{\Omega }%
\end{array}%
\end{equation*}%
and
\begin{equation*}
\begin{array}{l}
-\Delta _{q(x)}\overline{v}=\lambda ^{\bar{\sigma}}\geq \lambda ^{1+\frac{%
\bar{\sigma}\alpha _{2}^{+}}{p^{-}-1}}(c_{2})^{\alpha _{2}^{+}}(c_{0}\delta
)^{\beta _{2}^{-}} \\
\geq \lambda ^{1+\frac{\bar{\sigma}\alpha _{2}(x)}{p^{-}-1}}(c_{2})^{\alpha
_{2}(x)}(c_{0}\delta )^{\beta _{2}(x)}\geq \lambda \overline{u}^{\alpha
_{2}(x)}\overline{v}^{\beta _{2}(x)}\text{ \ in }\overline{\Omega },%
\end{array}%
\end{equation*}%
provided that $\lambda >0$ is sufficiently large. Consequently,
\begin{equation*}
\begin{array}{l}
\int_{\Omega }\left\vert \nabla \overline{u}\right\vert ^{p(x)-2}\nabla
\overline{u}\nabla \varphi \text{ }dx\geq \lambda \int_{\Omega }\overline{u}%
^{\alpha _{1}(x)}\overline{v}^{\beta _{1}(x)}\varphi \ dx\geq \lambda
\int_{\Omega }\overline{u}^{\alpha _{1}(x)}w^{\beta _{1}(x)}\varphi%
\end{array}%
\end{equation*}%
\begin{equation*}
\begin{array}{l}
\int_{\Omega }\left\vert \nabla \overline{v}\right\vert ^{q(x)-2}\nabla
\overline{v}\nabla \psi \text{ }dx\geq \lambda \int_{\Omega }\overline{u}%
^{\alpha _{2}(x)}\overline{v}^{\beta _{2}(x)}\psi \text{ }dx\geq \lambda
\int_{\Omega }\zeta ^{\alpha _{2}(x)}\overline{v}^{\beta _{2}(x)}\psi \text{
}dx,%
\end{array}%
\end{equation*}%
for $\lambda >0$ sufficiently large, $\zeta \in \lbrack \underline{u},%
\overline{u}]$, $w\in \lbrack \underline{v},\overline{v}]$ and $\left(
\varphi ,\psi \right) \in W_{0}^{1,p(x)}\left( \Omega \right) \times
W_{0}^{1,q(x)}\left( \Omega \right) $ with $\varphi ,\psi \geq 0$ a.e. in $%
\Omega $, showing that $\left( \overline{u},\overline{v}\right) $ is a
supersolution of (\ref{p}) for $\lambda >0$ large.
\end{proof}

We are now ready to prove our first main result.

\begin{proof}[Proof of Theorem \protect\ref{T1}]
By using (\ref{h1}), (\ref{h3}), (\ref{21}), (\ref{31}) and (\ref{31*}), we
get
\begin{equation*}
u^{\alpha _{1}(x)}v^{\beta _{1}(x)}\leq \underline{u}^{\alpha _{1}(x)}%
\overline{v}^{\beta _{1}(x)}\leq Cd(x)^{\alpha _{1}(x)}\text{ \ \ in }\Omega
\times \lbrack \underline{u},\overline{u}]\times \lbrack \underline{v},%
\overline{v}]
\end{equation*}%
and
\begin{equation*}
u^{\alpha _{2}(x)}v^{\beta _{2}(x)}\leq \overline{u}^{\alpha _{2}(x)}%
\underline{v}^{\beta _{2}(x)}\leq C^{\prime }d(x)^{\beta _{2}(x)}\text{ \ \
in }\Omega \times \lbrack \underline{u},\overline{u}]\times \lbrack
\underline{v},\overline{v}],
\end{equation*}%
where $C,C^{\prime }>0$ are constants. Then (\ref{h3}) enable us to apply
Theorem \ref{T2} and to conclude that there exists a positive solution $%
(u,v)\in C^{1,\nu }(\overline{\Omega })\times C^{1,\nu }(\overline{\Omega })$
of (\ref{p}), for some $\nu \in (0,1),$ within $[\underline{u},\overline{u}%
]\times \lbrack \underline{v},\overline{v}]$. This completes the proof.
\end{proof}

\section{Proof of Theorem \protect\ref{T12}}

\label{sec5}

For a fixed $\delta >0$ sufficiently small, let $u_{1}$ and $v_{1}$ be
solutions of the problems

\begin{equation}
-\Delta _{p(x)}u_{1}=\lambda ^{\sigma }\left\{
\begin{array}{ll}
w_{1}^{\alpha _{1}(x)} & \text{ in \ }\Omega \backslash \overline{\Omega }%
_{\delta } \\
d(x)^{\alpha _{1}(x)+\beta _{1}(x)} & \text{\ in \ }\Omega _{\delta }%
\end{array}%
\right. ,\text{ }u_{1}=0\text{ \ on }\partial \Omega  \label{40}
\end{equation}%
\begin{equation}
-\Delta _{q(x)}v_{1}=\lambda ^{\sigma }\left\{
\begin{array}{ll}
w_{2}^{\beta _{2}(x)} & \text{ in \ }\Omega \backslash \overline{\Omega }%
_{\delta } \\
d(x)^{\alpha _{2}(x)+\beta _{2}(x)} & \text{\ in \ }\Omega _{\delta }%
\end{array}%
\right. ,\text{ }v_{1}=0\text{ \ on }\partial \Omega  \label{41}
\end{equation}%
where $\Omega _{\delta }$ is defined by (\ref{70}) and $w_{1}$, $w_{2}$ are
solutions of problems (\ref{10}) with $\sigma >1$. Analysis similar to that
in the proof of Theorem \ref{T1}, namely by applying Hardy-Sobolev
Inequality and Minty-Browder Theorem, shows that $u_{1}$ and $v_{1}$ are
unique solutions of (\ref{40}) and (\ref{41}), respectively. On account of
Lemma \ref{L2}, $u_{1}$ and $v_{1}$ satisfy
\begin{equation}
\begin{array}{l}
\frac{w_{1}(x)}{2}\leq u_{1}(x)\quad \mbox{and}\quad \frac{w_{2}(x)}{2}\leq
v_{1}(x)\quad \mbox{in}\quad \overline{\Omega }.%
\end{array}
\label{81}
\end{equation}%
Moreover, similar arguments explored in the proof of \cite[Theorem 4.4]{QZ}
give $u_{1},v_{1}\in C(\overline{\Omega })$ and produce constants $%
c_{0},c_{1}>0,$ with $c_{0}:=c_{0}(\lambda ),$ $c_{1}:=c_{1}(\lambda )$,
such that%
\begin{equation}
\begin{array}{l}
u_{1}(x)\leq c_{0}d(x)^{\theta _{1}}\text{ \ and \ }v_{1}(x)\leq
c_{1}d(x)^{\theta _{2}}\text{ in }\Omega _{\delta },%
\end{array}
\label{82}
\end{equation}%
for some constants $\theta _{1},\theta _{2}\in (0,1),$ with $\theta
_{1},\theta _{2}\approx 1,$ and for $\delta >0$ small.

Let consider the functions $u_{0}$ and $v_{0}$ defined by
\begin{equation}
-\Delta _{p(x)}u_{0}=\left\{
\begin{array}{ll}
1 & \text{ in \ }\Omega \backslash \overline{\Omega }_{\delta } \\
-1 & \text{\ in \ }\Omega _{\delta }%
\end{array}%
\right. ,\text{ }u_{0}=0\text{ \ on }\partial \Omega  \label{83}
\end{equation}%
and
\begin{equation}
-\Delta _{q(x)}v_{0}=\left\{
\begin{array}{ll}
1 & \text{\ in \ }\Omega \backslash \overline{\Omega }_{\delta } \\
-1 & \text{\ in \ }\Omega _{\delta }%
\end{array}%
\right. ,\text{ }v_{0}=0\text{ \ on }\partial \Omega .  \label{83*}
\end{equation}%
According to \cite{YY} and Lemma \ref{L2}, it follows that
\begin{equation}
\begin{array}{l}
c_{3}\min \{\delta ,d(x)\}\leq u_{0}(x)\leq c_{4}\text{ \ and \ }%
c_{3}^{\prime }\min \{\delta ,d(x)\}\leq v_{0}(x)\leq c_{4}^{\prime }\text{
\ in }\Omega ,%
\end{array}
\label{84}
\end{equation}%
where $c_{3},c_{4},c_{3}^{\prime }$ and $c_{4}^{\prime }$ are positive
constants.

We claim that $\left( u_{1},v_{1}\right) \geq \left( u_{0},v_{0}\right) $ in
$\overline{\Omega }$. Indeed, by using (\ref{83}), (\ref{83*}), (\ref{66}), (%
\ref{67}), (\ref{h2}), (\ref{40}) and (\ref{41}), since $\sigma >1$, we have%
\begin{equation*}
\begin{array}{l}
-\Delta _{p(x)}u_{0}=\left\{
\begin{array}{ll}
1 & \text{ in \ }\Omega \backslash \overline{\Omega }_{\delta } \\
-1 & \text{\ in \ }\Omega _{\delta }%
\end{array}%
\right. \leq \lambda ^{\sigma }(C_{1}\lambda ^{\frac{\sigma }{p^{-}-1}%
})^{\alpha _{1}(x)}\leq \lambda ^{\sigma }w_{1}^{\alpha _{1}(x)} \\
\leq \lambda ^{\sigma }\left\{
\begin{array}{ll}
w_{1}^{\alpha _{1}(x)} & \text{if }d(x)>\delta \\
d(x)^{\alpha _{1}(x)+\beta _{1}(x)} & \text{if }d(x)<\delta%
\end{array}%
\right. =-\Delta _{p(x)}u_{1}\text{ \ in }\Omega%
\end{array}%
\end{equation*}%
and%
\begin{equation*}
\begin{array}{l}
-\Delta _{q(x)}v_{0}=\left\{
\begin{array}{ll}
1 & \text{ in \ }\Omega \backslash \overline{\Omega }_{\delta } \\
-1 & \text{\ in \ }\Omega _{\delta }%
\end{array}%
\right. \leq \lambda ^{\sigma }(C_{2}\lambda ^{\frac{\sigma }{q^{-}-1}%
})^{\beta _{2}(x)}\leq \lambda ^{\sigma }w_{2}^{\beta _{2}(x)} \\
\leq \lambda ^{\sigma }\left\{
\begin{array}{ll}
w_{2}^{\beta _{2}(x)} & \text{if }d(x)>\delta \\
d(x)^{\alpha _{2}(x)+\beta _{2}(x)} & \text{if }d(x)<\delta%
\end{array}%
\right. =-\Delta _{q(x)}v_{1}\text{ \ in }\Omega ,%
\end{array}%
\end{equation*}%
provided that $\lambda >0$ is large enough. Then the monotonicity of the
operators $-\Delta _{p(x)}$ and $-\Delta _{q(x)}$ leads to the conclusion.
The claim is proved.

The following result allows us to achieve useful comparison properties.

\begin{proposition}
\label{P1}Assume that (\ref{h4}) and (\ref{h2}) hold. Then, for $\lambda >0$
large enough, we have
\begin{equation}
-\Delta _{p(x)}u_{0}\leq \lambda u_{0}^{\alpha _{1}(x)}v_{0}^{\beta _{1}(x)}%
\text{ \ and \ }-\Delta _{q(x)}v_{0}\leq \lambda u_{0}^{\alpha
_{2}(x)}v_{0}^{\beta _{2}(x)}\text{ \ in }\overline{\Omega }  \label{36}
\end{equation}%
and
\begin{equation}
-\Delta _{p(x)}u_{1}\geq \lambda u_{1}^{\alpha _{1}(x)}v_{1}^{\beta _{1}(x)}%
\text{ \ and \ }-\Delta _{q(x)}v_{1}\geq \lambda u_{1}^{\alpha
_{2}(x)}v_{1}^{\beta _{2}(x)}\text{ \ in }\overline{\Omega }.  \label{36*}
\end{equation}
\end{proposition}

\begin{proof}
For all $\lambda >0$ we have
\begin{equation}
-u_{0}^{-\alpha _{1}(x)}v_{0}^{-\beta _{1}(x)}\leq 0<\lambda \text{ \ in }%
\Omega _{\delta },  \label{32}
\end{equation}%
\begin{equation}
-u_{0}^{-\alpha _{2}(x)}v_{0}^{-\beta _{2}(x)}\leq 0<\lambda \text{ \ in }%
\Omega _{\delta }.  \label{32*}
\end{equation}%
From (\ref{84}), (\ref{h2}) and (\ref{h4}), we have
\begin{equation}
u_{0}^{-\alpha _{1}(x)}v_{0}^{-\beta _{1}(x)}\leq c_{4}^{-\alpha
_{1}(x)}(c_{4}^{\prime })^{-\beta _{1}(x)}\leq \lambda \text{ \ in }\Omega
\backslash \overline{\Omega }_{\delta },  \label{35}
\end{equation}%
and
\begin{equation}
u_{0}^{-\alpha _{2}(x)}v_{0}^{-\beta _{2}(x)}\leq c_{4}^{-\alpha
_{2}(x)}(c_{4}^{\prime })^{-\beta _{2}(x)}\leq \lambda \text{ \ in }\Omega
\backslash \overline{\Omega }_{\delta },  \label{35*}
\end{equation}%
provided that $\lambda $ is sufficiently large. Then combining (\ref{32}) -
( \ref{35*}) together leads to (\ref{36}).

Now let us show (\ref{36*}). By (\ref{40}), (\ref{81}), (\ref{h2}), (\ref{h4}%
), (\ref{66}) and (\ref{67}), since $\sigma >1,$ one has
\begin{equation*}
\begin{array}{l}
u_{1}^{-\alpha _{1}(x)}v_{1}^{-\beta _{1}(x)}(-\Delta _{p(x)}u_{1})=\lambda
^{\sigma }u_{1}^{-\alpha _{1}(x)}v_{1}^{-\beta _{1}(x)}\left\{
\begin{array}{ll}
w_{1}^{\alpha _{1}(x)} & \text{if }d(x)\geq \delta \\
d(x)^{\alpha _{1}(x)+\beta _{1}(x)} & \text{if }d(x)<\delta%
\end{array}%
\right. \\
\geq \lambda ^{\sigma }(\frac{w_{1}}{2})^{-\alpha _{1}(x)}(\frac{w_{2}}{2}%
)^{-\beta _{1}(x)}\left\{
\begin{array}{ll}
w_{1}^{\alpha _{1}(x)} & \text{if }d(x)\geq \delta \\
d(x)^{\alpha _{1}(x)+\beta _{1}(x)} & \text{if }d(x)<\delta%
\end{array}%
\right. \\
\geq \lambda ^{\sigma }\left\{
\begin{array}{ll}
2^{\alpha _{1}(x)+\beta _{1}(x)}\delta ^{-\beta _{1}(x)} & \text{if }%
d(x)\geq \delta \\
(\frac{d(x)}{2})^{-\alpha _{1}(x)-\beta _{1}(x)}d(x)^{\alpha _{1}(x)+\beta
_{1}(x)} & \text{if }d(x)<\delta%
\end{array}%
\right. \\
\\
\geq \lambda ^{\sigma }2^{\alpha _{1}(x)+\beta _{1}(x)}\delta ^{-\beta
_{1}(x)}\geq \lambda \text{ \ in }\overline{\Omega },%
\end{array}%
\end{equation*}%
provided that $\lambda >0$ is sufficiently large. Similarly, since $\sigma
>1,$ from (\ref{41}), (\ref{81}), (\ref{h2}), (\ref{h4}), (\ref{66}) and (%
\ref{67}), we get
\begin{equation*}
\begin{array}{l}
u_{1}^{-\alpha _{2}(x)}v_{1}^{-\beta _{2}(x)}(-\Delta _{q(x)}v_{1})=\lambda
^{\sigma }u_{1}^{-\alpha _{2}(x)}v_{1}^{-\beta _{2}(x)}\left\{
\begin{array}{ll}
w_{2}^{\beta _{2}(x)} & \text{if }d(x)\geq \delta \\
d(x)^{\alpha _{2}(x)+\beta _{2}(x)} & \text{if }d(x)<\delta%
\end{array}%
\right. \\
\geq \lambda ^{\sigma }(\frac{w_{1}}{2})^{-\alpha _{2}(x)}(\frac{w_{2}}{2}%
)^{-\beta _{2}(x)}\left\{
\begin{array}{ll}
w_{2}^{\beta _{2}(x)} & \text{if }d(x)\geq \delta \\
d(x)^{\alpha _{2}(x)+\beta _{2}(x)} & \text{if }d(x)<\delta%
\end{array}%
\right. \\
\geq \lambda ^{\sigma }\left\{
\begin{array}{ll}
2^{\alpha _{2}(x)+\beta _{2}(x)}\delta ^{-\alpha _{2}(x)} & \text{if }%
d(x)\geq \delta \\
(\frac{d(x)}{2})^{-\alpha _{2}(x)-\beta _{2}(x)}d(x)^{\alpha _{2}(x)+\beta
_{2}(x)} & \text{if }d(x)<\delta%
\end{array}%
\right. \\
\geq \lambda ^{\sigma }2^{\alpha _{2}(x)+\beta _{2}(x)}\delta ^{-\alpha
_{2}(x)}\geq \lambda \text{ \ in }\overline{\Omega },%
\end{array}%
\end{equation*}%
provided that $\lambda >0$ is sufficiently large. This shows (\ref{36*}) and
ends the proof.
\end{proof}

Now we are ready to prove our second main result.

\begin{proof}[Proof of Theorem \protect\ref{T12}]
According to Proposition \ref{P1}, functions $(u_{0},v_{0})$ and $%
(u_{1},v_{1})$ verify the inequalities (\ref{6**}) and (\ref{7**}) in
Theorem \ref{T3}, respectively. In addition, by (\ref{h2}), (\ref{h4}) and (%
\ref{84}), we get
\begin{equation*}
u^{\alpha _{1}(x)}v^{\beta _{1}(x)}\leq u_{0}^{\alpha _{1}(x)}v_{0}^{\beta
_{1}(x)}\leq \check{C}d(x)^{\alpha _{1}(x)+\beta _{1}(x)}\quad \mbox{in}%
\quad \Omega
\end{equation*}%
\begin{equation*}
u^{\alpha _{2}(x)}v^{\beta _{2}(x)}\leq u_{0}^{\alpha _{2}(x)}v_{0}^{\beta
_{2}(x)}\leq \hat{C}d(x)^{\alpha _{2}(x)+\beta _{2}(x)}\quad \mbox{in}\quad
\Omega
\end{equation*}%
for $(u,v)\in \lbrack u_{0},u_{1}]\times \lbrack v_{0},v_{1}],$ where $%
\check{C},\hat{C}>0$ are constants. Then (\ref{h4}), (\ref{h2}), (\ref{h4**}%
), (\ref{82}) and (\ref{84}) allow to verify that the assumptions in the
Theorem \ref{T3} are satisfied. Thus, there exists a positive solution $%
(u,v)\in C^{1,\nu }(\overline{\Omega })\times C^{1,\nu }(\overline{\Omega })$
of (\ref{p}), for some $\nu \in (0,1),$ within $[u_{0},u_{1}]\times \lbrack
v_{0},v_{1}]$. This completes the proof.
\end{proof}

\begin{acknowledgement}
The work was accomplished while the second author was visiting the
University Federal of Campina Grande with CNPq-Brazil fellowship
402792/2015-7. He thanks for hospitality.
\end{acknowledgement}

\end{document}